\documentclass[11pt]{article}
\usepackage{fullpage}
\usepackage{amsfonts}
\usepackage{bbding}
\usepackage{graphicx,float}
\usepackage{amsfonts,amssymb,amsmath,amsthm,amscd}
\usepackage{mathrsfs}
\usepackage{xcolor}
\usepackage{empheq}
\usepackage{adforn}
\usepackage{cancel}
\usepackage{xr-hyper}
\usepackage{xr}
\usepackage{mdframed}
\usepackage{mathdots}
\usepackage{enumerate}
\usepackage{lmodern}
\usepackage{mdframed}
\usepackage{stmaryrd}
\usepackage{blindtext}
\usepackage[all, knot]{xy}
\usepackage{leftidx}
\usepackage{accents}
\usepackage{appendix}
\usepackage{cases}
\usepackage{bbm}
\usepackage{subfigure}

\definecolor{slightblue}{rgb}{.8, .8, 1}
\definecolor{hair}{RGB}{100,225,190}
\definecolor{ruby}{RGB}{220,50,120}
\definecolor{grass}{RGB}{150,220,110}
\definecolor{ceruleanblue}{rgb}{0.16, 0.32, 0.75}
\definecolor{deepcarmine}{rgb}{0.66, 0.13, 0.24}
\definecolor{otterbrown}{rgb}{0.4, 0.26, 0.13}
\definecolor{sapphire}{rgb}{0.03, 0.15, 0.4}

\usepackage[colorlinks=true,  linkcolor=otterbrown, citecolor=sapphire,  urlcolor=hair!80!black]{hyperref}

\usepackage[T1]{fontenc}

\newtheorem{theorem}{Theorem}[section] 
\newtheorem{lemma}[theorem]{Lemma}
\newtheorem{proposition}[theorem]{Proposition}

\theoremstyle{definition} 
\newtheorem{definition}[theorem]{Definition}

\newtheorem{expl}{Exploration process}
\newtheorem{ass}{Assumption}

\newtheorem{remark}[theorem]{Remark} \numberwithin{equation}{section}
\numberwithin{figure}{section}

\newcommand{\one}{\mathbf{1}}

\newcommand{\Cb}{\mathbb{C}}
\newcommand{\Db}{\mathbb{D}}
\newcommand{\Eb}{\mathbb{E}}

\newcommand{\Hb}{\mathbb{H}}

\newcommand{\Pb}{\mathbb{P}}

\newcommand{\Rb}{\mathbb{R}}

\newcommand{\Ac}{\mathcal{A}}

\newcommand{\Cc}{\mathcal{C}}
\newcommand{\Dc}{\mathcal{D}}
\newcommand{\Ec}{\mathcal{E}}
\newcommand{\Fc}{\mathcal{F}}
\newcommand{\Gc}{\mathcal{G}}
\newcommand{\Hc}{\mathcal{H}}

\newcommand{\Lc}{\mathcal{L}}

\newcommand{\Qc}{\mathcal{Q}}

\newcommand{\Wc}{\mathcal{W}}

\usepackage[section]{placeins}
\usepackage{wrapfig}
\usepackage{lpic}

\newcommand{\wt}{\widetilde}
\newcommand{\wh}{\widehat}
\newcommand{\ol}{\overline}

\DeclareFontFamily{OMX}{yhex}{}
\DeclareFontShape{OMX}{yhex}{m}{n}{<->yhcmex10}{}
\DeclareSymbolFont{yhlargesymbols}{OMX}{yhex}{m}{n}
\DeclareMathAccent{\wideparen}{\mathord}{yhlargesymbols}{"F3}

\makeatletter
\newcommand*\rel@kern[1]{\kern#1\dimexpr\macc@kerna}
\newcommand*\wb[1]{%
	\begingroup
	\def\mathaccent##1##2{%
		\rel@kern{0.8}%
		\overline{\rel@kern{-0.8}\macc@nucleus\rel@kern{0.2}}%
		\rel@kern{-0.2}%
	}%
	\macc@depth\@ne
	\let\math@bgroup\@empty \let\math@egroup\macc@set@skewchar
	\mathsurround\z@ \frozen@everymath{\mathgroup\macc@group\relax}%
	\macc@set@skewchar\relax
	\let\mathaccentV\macc@nested@a
	\macc@nested@a\relax111{#1}%
	\endgroup
}
\makeatother

\def \eps {\varepsilon}

\newcommand{\dist}{\mathrm{dist}}
\newcommand{\ind}{\mathbf{1}}

\newcommand{\ba}{\Ac^+}

\newcommand{\bone}{\Ec_1^+}
\newcommand{\btwo}{\Ec_2^+}
\newcommand{\ione}{\Ec_1}
\newcommand{\itwo}{\Ec_2}

\newcommand{\gbtwo}{\Gc_2^+}

\newcommand{\gitwo}{\Gc_2}
\newcommand{\bxi}{\xi^+}

\newcommand{\ionec}{\ione^{\mathrm{C}}}

\title{Up-to-constants estimates on four-arm events\\ for simple conformal loop ensemble}
\author{Yifan Gao \and Pierre Nolin \and Wei Qian}
\date{\today}
\begin{document}
	\maketitle
	
	\begin{abstract}
We prove up-to-constants estimates for a general class of four-arm events in simple conformal loop ensembles, i.e.\ CLE$_\kappa$ for $\kappa\in (8/3,4]$. 
The four-arm events that we consider can be created by either one or two loops, with \emph{no constraint} on the topology of the crossings.
Our result is a key input in our series of works \cite{GNQ1,GNQ2} on percolation of the two-sided level sets in the discrete Gaussian free field (and level sets in the occupation field of the random walk loop soup).

In order to get rid of all constraints on the topology of the crossings, we rely on the Brownian loop-soup representation of simple CLE \cite{MR2979861}, and a ``cluster version'' of a separation lemma for the Brownian loop soup.
As a corollary, we also obtain up-to-constants estimates for a general version of four-arm events for SLE$_\kappa$  for $\kappa\in (8/3,4]$. This fixes (in the case of four arms and $\kappa\in(8/3,4]$) an essential gap in \cite{MR3846840} and improves some estimates therein.

	\end{abstract}
	
	\tableofcontents

\section{Introduction}

The conformal loop ensemble (CLE) is introduced in \cite{MR2494457} as a universal candidate for the scaling limit of a large class of discrete statistical physics models at criticality, e.g.\ Bernoulli percolation \cite{Sm2001, CN2006}, the Ising model \cite{CS2012, BH2019} and the random-cluster model with cluster weight $q=2$ (also known as FK Ising) \cite{Sm2010, KS2019, KS2016}.
For $\kappa\in(8/3,4]$, CLE$_\kappa$ in a simply connected domain $D\subsetneq \Cb$ is almost surely (a.s.) a countable collection of disjoint simple loops. For $\kappa\in (4, 8)$, the loops in CLE$_\kappa$ are a.s.\ non-simple and can touch each other.
In this paper, we work on the regime of simple CLE, i.e.\ where $\kappa \in (8/3,4]$.
Many basic properties of simple CLE are established in the foundational work \cite{MR2979861}.  In particular, CLE$_\kappa$ was constructed there using the outer boundaries of the outermost clusters in the Brownian loop soup with intensity $\alpha\in(0,1/2]$, where
\begin{align}\label{eq:c_kappa}
\alpha= (3\kappa-8)(6-\kappa)/(4\kappa).
\end{align}

In this paper, we focus on estimates on four-arm events in a simple CLE, in particular with a view to applying them to statistical physics models converging to such a CLE. More specifically, our aim is to derive, from these estimates, a qualitative description of the set of contact edges between clusters.
For example, the clusters of a subcritical or a critical random walk loop soup (RWLS) are known to converge to simple CLE \cite{BCL2016,Lu2019}, and the upper bounds on four-arm probabilities for CLE that we obtain in this paper serve as an input  in \cite{GNQ1}, where we derive various properties in the discrete setting, on four-arm events in the RWLS. The points where four arms occur in the RWLS correspond exactly to the contact edges between two clusters. 
Plugging in the four-arm estimates for the RWLS into a summation argument over all possible configurations of contact edges between the clusters, we prove in \cite{GNQ2} the existence of a phase transition for percolation of the level sets in the occupation field of a subcritical RWLS, and an analogous result at criticality. This analysis, via an isomorphism between the (critical) RWLS and the discrete Gaussian free field (GFF) \cite{MR2815763}, then yields corresponding results for two-sided level sets in the discrete GFF.

\smallskip

Arm exponents often happen to be an important quantity to analyze statistical physics models. The polychromatic arm exponents for Bernoulli percolation (and the one-arm exponent) have been first predicted by theoretical physics methods (see in particular \cite{ADA99}, and the references therein). They were established rigorously in \cite{SW2001}, shortly after the introduction of the Schramm-Loewner evolutions (SLE) \cite{MR1776084}.
More precisely, in that paper, Smirnov and Werner computed the arm exponents of SLE$_\kappa$ for $\kappa=6$, and then they appealed to the relation between SLE$_6$ and planar Bernoulli percolation at criticality \cite{Sm2001}.  
The proof in \cite{SW2001} relies on previous works \cite{MR1879850,MR1879851,MR1899232}, which computed other closely related exponents for SLE$_\kappa$ for $\kappa\in(0,8)$ by using suitable SLE martingales. 
The values of the  SLE arm exponents for all $\kappa\in (0,8)$ were also derived by physicists through the KPZ relation, see e.g.\ \cite{MR2112128,MR2581884}: 
For $j\ge 1$, the interior $2j$-arm exponent $\xi_{2j}$ and the boundary $j$-arm exponent $\bxi_j$ are given by
\begin{align}\label{eq:arm_exp}
 \xi_{2j}= \frac{16j^2-(\kappa-4)^2}{8\kappa}, \quad \bxi_{2j}= \frac{j(4j+4-\kappa)}{\kappa}, \quad \bxi_{2j-1}=\frac{(j-1)(4j+4-\kappa)}{\kappa}.
\end{align}
Up-to-constants estimates on the interior and boundary two-arm events for SLE$_\kappa$ were respectively obtained in  \cite{MR2435854} and \cite{MR2574734}, for all $\kappa\in(0,8)$, yielding (resp.), among other things, the Hausdorff dimension of SLE$_\kappa$ and a quantitative description of the proximity of SLE$_\kappa$ to the boundary.
Later on, it was proved in \cite{MR3846840} for $\kappa\in (0,4)$ (also see similar results  in \cite{MR3718717,MR3768961} for $\kappa\in (4,8)$) that $\Pb(\wt A_j)\asymp \eps^{\bxi_j}$ and $\Pb(A_{2j})= \eps^{\xi_{2j}+o(1)}$ as $\eps\to 0$, where $\wt A_j$ (resp.\ $A_{2j}$) stands for some specific boundary $j$-arm (resp.\ interior $2j$-arm)  event, namely that the SLE curve makes $j$ (resp.\ $2j$) arms, following a \emph{particular order and topology}, between $\partial B_\eps(z)$ and a given arc at macroscopic distance from $z$, where $z$ is a boundary (resp.\ interior) point in the domain. Here, the symbol $\asymp$ means that the ratio between the two sides remains bounded away from $0$ and $\infty$, and $B_\eps(z)$ denotes the open ball with radius $\eps$ centered on $z$ (in the following, we drop $z$ from the notation when $z=0$).
Thanks to the convergence of the Ising (resp.\ FK Ising) interfaces to SLE$_3$ (resp.\ SLE$_{16/3}$) \cite{MR3151886,MR3073886}, the estimates from \cite{MR3846840,MR3718717,MR3768961} on SLE arm exponents then yield \cite{MR3846840,MR3768961} the arm exponents for the critical Ising and FK Ising models.

Nevertheless, we have identified an essential gap in the proof of the upper bound $\Pb(A_{2j}) \leq \eps^{\xi_{2j}+o(1)}$ for the interior $2j$-arm event in \cite[Proposition 4.1]{MR3846840}, for $\kappa\in(0,4)$ and $j\ge 2$ 
(similar gaps also exist in the proofs of the upper bounds in \cite[Proposition 3]{MR3768961}, on three different types of interior SLE$_\kappa$ arm events, for $\kappa\in(4,8)$).
More precisely, we believe that the proof of this upper bound needs not only an estimate on a well-ordered boundary arm event as an input (as it is claimed there), but actually needs to rely on an estimate for a rather general boundary arm event, which is lacking in that paper. 
On the other hand, it seems that the proofs in \cite{MR3846840,MR3718717,MR3768961} on SLE arm events are highly dependent on the specific order and topology of the crossings, so they should require additional nontrivial arguments to get rid of these constraints. We refer to Remark~\ref{rmk:wu} for more details.

\smallskip

In this paper, we derive \emph{up-to-constants} estimates for a general form of four-arm events, with no condition on the order of crossings, for both the boundary case and the interior case, and for both CLE$_\kappa$ and SLE$_\kappa$, for $\kappa\in (8/3, 4]$.
A chordal SLE has two endpoints on the boundary of the domain, hence corresponds to the interface in a statistical physics model with Dobrushin boundary conditions. In contrast, the CLE describes the loop interfaces inside a model with ``homogeneous'' boundary conditions.
Since a CLE$_\kappa$ loop looks locally like an SLE$_\kappa$ curve, it is natural to expect that the arm exponents for CLE$_\kappa$ coincide with those for SLE$_\kappa$.
Once we establish the estimates for CLE$_\kappa$, we will immediately deduce analogous results for SLE$_\kappa$.

One can also consider the multiple SLE's, which correspond to the multiple interfaces in a model with alternating boundary conditions, see e.g.\ \cite{MR2358649, MR2187598, MR2310306}. In \cite{MR4082185,MR4235483}, Zhan has computed the Green's functions related to a system of $2$-SLE, where both SLE curves in the $2$-SLE go through a given ball $B_\eps(z)$, thus creating four arms. These estimates, recalled as Theorems~\ref{thm:zhan} and~\ref{thm:zhanb} below, are instrumental in our proofs.
On the other hand, the four arms in \cite{MR4082185,MR4235483} are created by two distinct SLE curves with four separate endpoints. In this paper, we work with the four-arm events for the CLE (as well as single SLE curves), for which the separation is not incorporated in the definition.

In order to treat all possible topologies of arm crossings, the main technical difficulty of this work is to devise and prove a well-suited separation lemma. Such separation properties have been established in various models, and play an important role in their analysis. It seems that the earliest version was obtained in Kesten's famous work \cite{Ke1987a} on two-dimensional percolation near criticality. Later, some versions for non-intersecting Brownian motions were shown by Lawler \cite{MR1386294,La1998}. There were then a great deal of generalizations afterwards; see e.g. \cite{No2008,GPS2013,du2022sharp,Ma2009,MR4291423,GLQ2022} for a non-exhaustive list of references. We postpone a more detailed discussion to Section~\ref{subsubsec3}.

\smallskip
We are now ready to state our up-to-constants estimates on the four-arm events for simple CLE (Section~\ref{subsubsec1}), which, as we just explained, rely on a key separation lemma (Section~\ref{subsubsec3}). As a consequence, we then derive up-to-constants estimates on the four-arm events for SLE (Section~\ref{subsubsec2}).

\subsection{Four-arm events for simple CLE}\label{subsubsec1}

For $\kappa\in(8/3,4]$, let $\Gamma^+$ (resp.\ $\Gamma$) be a CLE$_\kappa$ in the upper half-plane $\Hb$ (resp. the unit disk $\Db$). 
Let us first define the following general four-arm events, which correspond to the existence of four curves crossing a given annulus in the CLE$_\kappa$.
\begin{itemize}
\item (Arms by one loop) Let $\bone(\eps,r)$ (resp.\ $\ione(\eps,r)$)  be the event that there is a loop $\gamma$  in $\Gamma^+$ (resp.\ $\Gamma$) which makes $4$ crossings in the annulus $A_{\eps,r}:= B_r\setminus \overline{B_\eps}$, in other words $\gamma \cap \overline{A_{\eps,r}}$ contains at least $4$ curves, which each have one endpoint on $\partial B_r$ and one endpoint on $\partial B_\eps$.
\item (Arms by two loops) Let $\btwo(\eps,r)$ (resp.\ $\itwo(\eps,r)$) be the event that there are at least two loops $\gamma_1$ and $\gamma_2$ in $\Gamma^+$  (resp.\ $\Gamma$) such that each of them intersects both $\partial B_\eps$ and $\partial B_r$. 
\end{itemize}
We define the \emph{boundary four-arm event} by $\Ac_4^+(\eps, r): = \bone(\eps, r) \cup \btwo(\eps, r)$, and the \emph{interior four-arm event} by $\Ac_4(\eps, r):= \ione(\eps, r) \cup \itwo(\eps, r)$.
See Figure~\ref{fig:4arm_CLE} for an illustration of the possible configurations for the interior four-arm event.
\begin{figure}[b]
	\centering
	\includegraphics[width=0.85\textwidth]{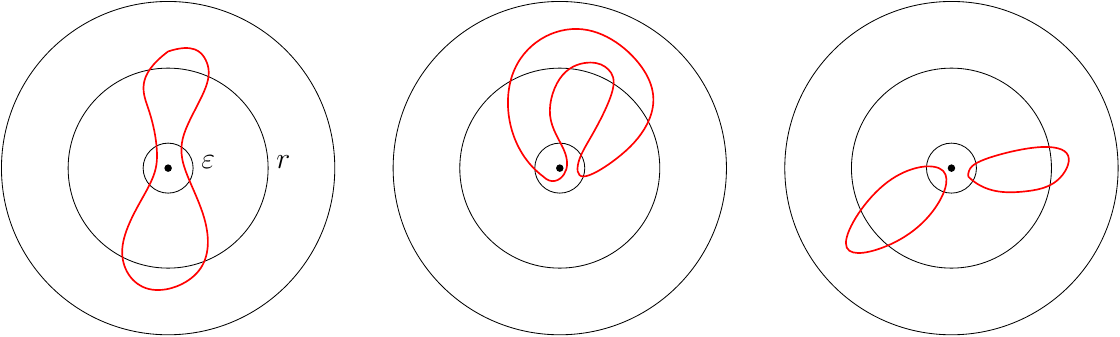}
	\caption{We depict three possible configurations of $\Ac_4(\eps, r)$. The left and middle figures belong to the event $ \ione(\eps, r)$. The right figure belongs to the event $\itwo(\eps, r)$.}
	\label{fig:4arm_CLE}
\end{figure}

Our first main result consists in the following up-to-constants estimates for four-arm events in CLE$_\kappa$.
\begin{theorem}\label{thm:main}
For $\kappa\in(8/3,4]$ and $r\in(0,1)$, we have the following estimates, as $\eps \to 0$,
\begin{align}
\label{eq:b4arm}
& \Pb(\bone(\eps, r)) \asymp \Pb(\btwo(\eps, r)) \asymp \Pb(\ba_4(\eps, r)) \asymp \eps^{\xi^+_{4}(\kappa)}, \\
\label{eq:i4arm}
& \Pb(\ione(\eps, r)) \asymp \Pb(\itwo(\eps, r)) \asymp \Pb(\Ac_4(\eps, r)) \asymp \eps^{\xi_{4}(\kappa)},
\end{align}
where
$\xi^+_4(\kappa)= 2(12-\kappa)/\kappa$ and $\xi_4(\kappa)=(12-\kappa)(\kappa+4)/(8\kappa)$.
\end{theorem}

Let us comment briefly on the strategy of the proof of  Theorem~\ref{thm:main}. First, we make use of a Markovian exploration of the CLE, in the style of \cite{MR2979861}. We illustrate such explorations in our warm-up Section~\ref{sec:2arm}, where we derive up-to-constants estimates in the simpler situation of two-arm events in CLE. 
Secondly, the above-mentioned estimates on four-arm probabilities for $2$-SLE systems established by Zhan in \cite{MR4082185,MR4235483} (see Theorems~\ref{thm:zhan} and~\ref{thm:zhanb} below) play a central role. 
Finally, in our setting of CLE, we rely crucially on a separation lemma, which we discuss in the next subsection.

\subsection{Separation lemma}\label{subsubsec3}

We derive a separation lemma for CLE loops. Roughly speaking, it states that two crossing loops in a given annulus have a uniformly positive chance to get well-separated, e.g. near the outer boundary of the annulus. More specifically, consider CLE in $\Db$. On the event $\itwo(\eps, r)$, there are two loops $\gamma_1$ and $\gamma_2$ that cross the annulus $A_{\eps,r}$, and we can further define $\wt\itwo(\eps, r)$ to be the subevent that $\itwo(\eps, r)$ occurs with the additional requirement that $\gamma_j \subset (B_r \cup B_{r/10}(re^{i(j-1)\pi})) \setminus B_{r/10}(re^{i j \pi})$ for both $j=1,2$ (see Figure~\ref{fig:separation} for an illustration of that event). We prove the following.

\begin{figure}[b]
	\centering
	\includegraphics[width=0.33\textwidth]{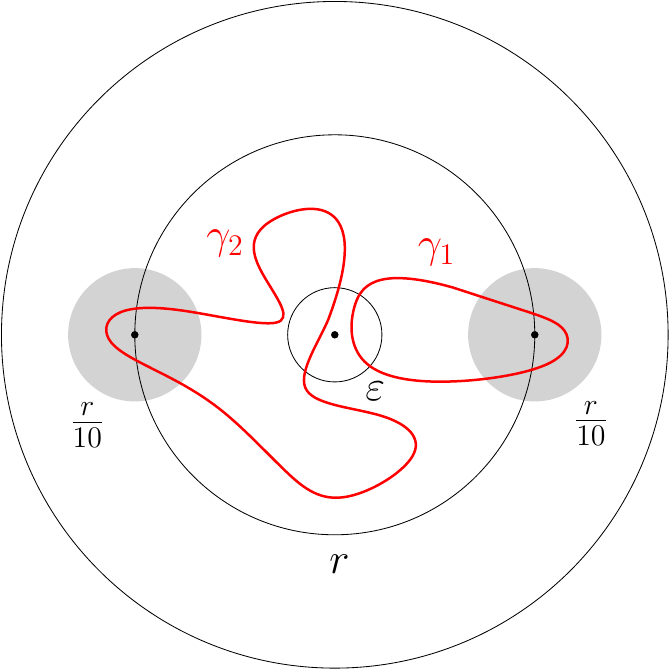}
	\caption{The well-separated arm event $\wt\itwo(\eps, r)$.}
	\label{fig:separation}
\end{figure}

\begin{theorem}\label{thm:sep_cle}
	For $\kappa\in(8/3,4]$, $0<2\eps<r<1/2$, we have
	\[
	\Pb(\itwo(\eps, r)) \lesssim \Pb(\wt\itwo(\eps, r)),
	\]	
	where the symbol $\lesssim$ involves an implicit constant which only depends on $\kappa, r$.
\end{theorem}

Our proof relies on the correspondence \cite{MR2979861} between CLE$_\kappa$ and the outer boundaries of the outermost loop clusters in the Brownian loop soup (BLS) with intensity $\alpha$, where $\alpha$ is related to $\kappa$ by \eqref{eq:c_kappa}.
As we will see in the proof, separation lemmas of this type turn out to be more complicated than the ones established in \cite{GLQ2022,GNQ1} for loop soups. In latter cases, one usually starts with some Brownian motions or random walks, which can be used to keep track of the configuration. For example, in Lemma~\ref{lem:sep-jk}, we consider two Brownian motions crossing an annulus inside an independent BLS, and we study the non-intersection event associated with the Brownian motions together with the loop-soup clusters that they hook. Then, there is a natural way to describe how close they are, and how likely it is that they can be ``nicely'' extended to the next scale without intersecting, through some geometric quantity (see \eqref{eq:Qjk}), which is commonly called \emph{quality}. Then, the separation lemma reduces to analyzing the quality across scales, which is now a streamlined method. In order to use a similar method to prove Theorem~\ref{thm:sep_cle}, the first task is to find a suitable notion of quality associated with loop-soup clusters. However, in that case, it seems difficult to express the quality in some \emph{purely} geometric way, as in \eqref{eq:Qjk}. We found a well-suited definition of quality, see \eqref{eq:quality} below, which combines geometric and probabilistic information simultaneously (in terms of extension probabilities through ``good'' pairs of Brownian loops). By analyzing carefully this object (in combination with an additional stability result on arm events, Lemma~\ref{lem:quasi}), we can then adapt the standard framework to establish separation, and thus show Theorem~\ref{thm:sep_cle}.

\subsection{Four-arm events for SLE}\label{subsubsec2}
As a corollary of our estimates for CLE$_\kappa$, we also deduce analogous results on a general form of four-arm events for SLE$_\kappa$, defined in the following. 
\begin{itemize}
\item (Boundary four-arm event for SLE) 
Let $\eta$ be a chordal SLE$_\kappa$ in $\Hb$ from $0$ to $\infty$. For $1>r>\eps>0$, let $\tau_1$ be the first time that $\eta$ hits $B_\eps(1)$, let $\sigma_1$ be the first time after $\tau_1$ that $\eta$ hits $\partial B_r(1)$, and let $\tau_2$ be the first time after $\sigma_1$ that $\eta$ hits $B_\eps(1)$. We introduce
\begin{align}\label{eq:def_4arm_b}
\Wc^+_4 (\eps,r):=\{\tau_2<\infty\}.
\end{align}

\item (Interior four-arm event for SLE) 
Let $\eta$ be a chordal SLE$_\kappa$ in $\Db$ from $1$ to $a\in \partial \Db \setminus \{1\}$. For $1>r>\eps>0$, let $\tau_1$ be the first time that $\eta$ hits $B_\eps$, let $\sigma_1$ be the first time after $\tau_1$ that $\eta$ hits $\partial B_r$, and let $\tau_2$ be the first time after $\sigma_1$ that $\eta$ hits $B_\eps$. We introduce
\begin{align}\label{eq:def_4arm_int}
\Wc_4 (a, \eps, r):=\{\tau_2<\infty\}.
\end{align}
\end{itemize}

\begin{theorem}\label{thm:sle_int_utc}
For $\kappa\in(8/3,4]$, $1>r>\eps>0$ and $a\in \partial \Db \setminus \{1\}$, we have, as $\eps\to 0$,
\begin{align}
\label{eq:sle1}
&\Pb[\Wc^+_4 (\eps,r)] \asymp \eps^{\bxi_4(\kappa)},\\
\label{eq:sle2}
&\Pb[\Wc_4 (a, \eps, r)] \asymp \eps^{\xi_4(\kappa)}.
\end{align}
The implicit constants in \eqref{eq:sle1} (resp.\ \eqref{eq:sle2}) only depend on $\kappa, r$ (resp.\  $\kappa, a, r$).
\end{theorem}

\begin{figure}[b]
	\centering
	\includegraphics[width=\textwidth]{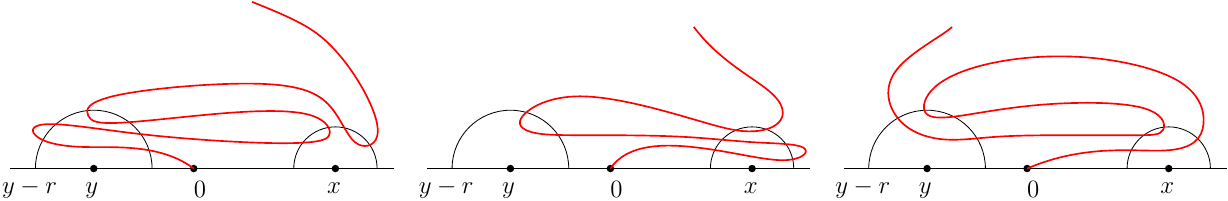}
	\caption{We depict three simple cases of boundary four-arm events (the two half-disks, centered on $x$ and $y$, have respective radii $\eps$ and $r$). None of them belongs to the event $ \Hc^\alpha_3(\eps, x, y,r)$ or $ \Hc^\alpha_4(\eps, x, y,r)$ considered in \cite{MR3846840}: roughly speaking, the curve does not hit the right arcs in the right order (see \eqref{eq:h1} and \eqref{eq:h2} for precise definitions). Our Theorem~\ref{thm:sle_int_utc} implies in particular that their probabilities are upper-bounded by a constant times $\eps^{\bxi_4(\kappa)}$.}
	\label{fig:4arm_bdy}
\end{figure}

\begin{figure}
	\centering
	\includegraphics[width=\textwidth]{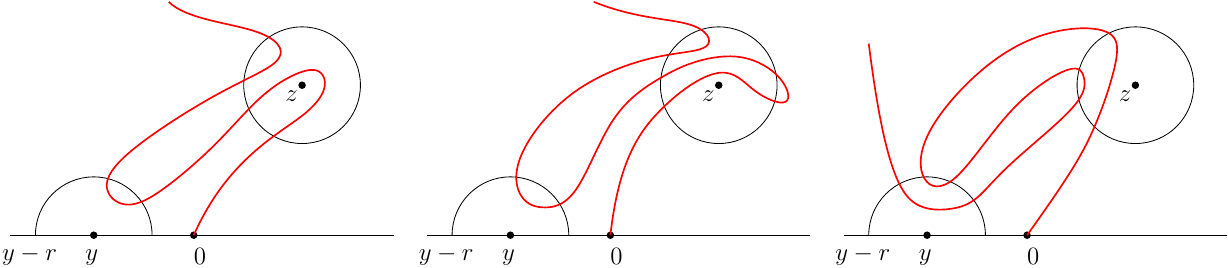}
	\caption{We depict three simple cases of interior four-arm events (where the disk around $z$ and the half-disk around $y$ have radii $\eps$ and $r$, respectively). None of them belongs to the event $ \Ec_4(\eps, z, y,r)$ considered in \cite{MR3846840} (see \eqref{eq:int_arm_event} for a precise definition). Our Theorem~\ref{thm:sle_int_utc} implies, after a conformal map from $\Hb$ onto $\Db$, that the probabilities of these events are upper-bounded by a constant times $\eps^{\xi_4(\kappa)}$.}
	\label{fig:4arm_event}
\end{figure}

We would like to make the following remarks, to compare our results to previous results on SLE arm exponents obtained in \cite{MR3846840}.
\begin{itemize}
\item We recall in Section~\ref{subsec:comparison} the precise definitions of the boundary $2j$-arm events $\Hc^\alpha_{2j-1}(\eps, x, y, r)$, $\Hc^\alpha_{2j}(\eps, x, y, r)$ and the interior $2j$-arm events $\Ec_{2j}(\eps, z, y, r)$ considered in \cite{MR3846840}. These definitions are very restrictive on the topology of the crossings, in the sense that they need to hit exactly a specific arc each time.
For instance, for the case of four-arm events, Theorem~\ref{thm:sle_int_utc} implies up-to-constant upper bounds on the events in Figures~\ref{fig:4arm_bdy} and~\ref{fig:4arm_event}, which were \emph{not} considered in \cite{MR3846840}.
On the other hand, the lower bounds in \eqref{eq:sle1} and \eqref{eq:sle2} can be deduced from  \cite{MR3846840}.

\item In fact, it is this specific upper bound that was obtained in \cite[Proposition 4.1]{MR3846840} in the interior case:
\begin{align*}
\Pb[\Ec_{2j}(\eps, z, y,r) \cap \Fc] \leq \eps^{\xi_{2j}+o(1)},
\end{align*}
for some well-suited event $\Fc$, involving a non-explicit quantity $R$ (see Proposition~\ref{prop:sle_int} below). As we explain in Remark~\ref{rmk:wu}, we fix (in the case $j=2$ and $\kappa\in(8/3,4]$) an essential gap in the proof of that result. Furthermore, we give a better upper bound, which is up-to-constant.

\end{itemize}

\subsection{Organization of the paper}

After collecting some useful earlier results in Section~\ref{sec:prelim}, we show estimates on two-arm probabilities in Section~\ref{sec:2arm}, as a warm-up for the more complicated situation of four arms. We then derive separation properties in Section~\ref{sec:separation}, where we prove a strenghtened version of Theorem~\ref{thm:sep_cle}, which is suitable for our later proofs (Proposition~\ref{prop:sep_i}). In that section, which is the core of our paper, we also deduce the equivalence of various four-arm events from the separation result. We then use these tools to establish up-to-constants estimates for four-arm events in CLE (Theorem~\ref{thm:main}) in Section~\ref{sec:est_CLE}, before obtaining our general four-arm estimates for SLE (Theorem~\ref{thm:sle_int_utc}) in Section~\ref{sec:est_SLE}.

\section{Preliminaries} \label{sec:prelim}
In this section, we recall a few results that will be useful later on. In the whole paper, we let $\Hb:=\{ z \in \Cb : \Im z >0 \}$ be the upper half-plane, and we let $\Db$ denote the unit disk (with center $0$ and radius $1$). For any $z \in \Cb$, let $B_r(z) := \{z' \in \Cb : |z-z'| < r\}$ be the ball with radius $r > 0$ centered on $z$, and $A_{r,R}(z):=\{ z' \in \Cb : r < |z-z'| < R \}$ be the annulus with radii $0 < r < R$ around $z$. Furthermore, we write $B_r = B_r(0)$ and $A_{r,R} = A_{r,R}(0)$. If a set $C$ intersects both $\partial B_r(z)$ and $\partial B_R(z)$, we say that it \emph{crosses} the annulus $A_{r,R}(z)$.

\begin{lemma}[Koebe $1/4$ in the upper half-plane]\label{lem:Koebe}
Let $K\subset \ol\Hb$ be a compact such that $H:=\Hb\setminus K$ is simply connected.
Let $f$ be a conformal map from $H$ onto $\Hb$ with $f(\infty) =\infty$.
For $z\in \Rb$ and $r>0$
such that $B_r(z) \cap \Hb \subseteq H$, we have
\begin{align}\label{eq:kb2}
B_{f'(z) r/4}(f(z)) \cap \Hb \subseteq f(B_r(z) \cap \Hb) \subseteq B_{4f'(z) r} (f(z)) \cap \Hb.
\end{align}
\end{lemma}
\begin{proof}
By the Schwarz reflection principle, we can extend $f$ to a bijective conformal map from $D=\Cb\setminus (K \cup \{\bar z \mid z\in K\})$ to $f(D)$. The lemma then follows from Koebe $1/4$ theorem applied to $D$ and the extended map $f$.
\end{proof}

We will need the following result from  \cite{MR2435854} (2002), which establishes $\xi_2$ as the  interior two-arm exponent for SLE$_\kappa$. A version of this result was first established in  \cite[Proposition 4]{MR2435854} in the upper half-plane, but later we will in fact use the following version for general domains.

\begin{proposition}[Corollary 5, \cite{MR2435854}]
\label{prop:beffara}
Let $D\subsetneq \Cb$ be a simply connected domain, and $a, b$ be two points on $\partial D$. 
Fix $\kappa\in(0,8)$. Let $\eta$ be an SLE$_\kappa$ in $D$ from $a$ to $b$. For all $z\in D$ and $\eps< d(z, \partial D)/2$, we have, as $\eps\to 0$,
\begin{align*}
\Pb(\eta \cap B_\eps(z) \not=\emptyset) \asymp \left(\frac{\eps}{d(z, \partial D)}\right)^{\xi_2(\kappa)} (\omega_z(ab) \wedge \omega_z(ba))^{\bxi_2(\kappa)},
\end{align*}
where $\omega_z$ is the harmonic measure on $\partial D$ seen from $z$, and $ab$ is the positively oriented arc from $a$ to $b$ along $\partial D$.
\end{proposition}

We will also need the following result from \cite{MR2574734} (2007), which establishes $\bxi_2$ as the boundary two-arm exponent for SLE$_\kappa$.
The statement that we cite below follows immediately from \cite[Lemma 2.1, Lemma 2.2 and Proposition 2.3]{MR2574734}.
\begin{proposition}[\cite{MR2574734}]
\label{prop:schramm}
Fix $\kappa\in(0,8)$. Let $\eta$ be an SLE$_\kappa$ in $\Hb$ from $0$ to $\infty$. For $x>0$, as $\eps\to 0$,
\begin{align*}
\Pb(\eta \cap B_\eps(x)\not=\emptyset) \asymp (\eps/x)^{\bxi_2(\kappa)}.
\end{align*}
\end{proposition}

We now recall two theorems from \cite{MR4082185, MR4235483} on the Green's function of 2-SLE. The original statements contain more properties, but we only state two simplified versions which are enough for our purpose.
In the two theorems below, let $D$ be a simply connected domain with four distinct boundary points (prime ends) $a_1, b_1, a_2, b_2$ such that $a_1$ and $a_2$ together separate $b_1$ from $b_2$ on $\partial D$. 
Let $(\eta_1, \eta_2)$ be a $2$-SLE$_\kappa$ in $D$ with link pattern $(a_1, b_1; a_2, b_2)$.
More precisely, this means that $(\eta_1, \eta_2)$ is distributed as a pair of independent SLE's in $D$ ($\eta_i$ between $a_i$ and $b_i$, for $i=1,2$) conditioned not to intersect each other.

\begin{theorem}[Theorem 1.1, \cite{MR4082185}]\label{thm:zhan}
Let $\kappa\in(0,8)$. Let $z_0\in D$. 
There exists $G(\kappa; D; a_1, b_1; a_2, b_2; z_0) \in (0,\infty) $  such that as $\eps\to 0$,
\begin{align*}
\Pb\left[\dist(z_0, \eta_i)< \eps, i=1,2 \right] = G(\kappa; D; a_1, b_1; a_2, b_2; z_0) \eps^{\xi_4} (1+o(1)).
\end{align*}
\end{theorem}

\begin{theorem}[Theorem 1.1, \cite{MR4235483}]\label{thm:zhanb}
Let $\kappa\in(0,4]$. Let $z_0\in \partial D\setminus\{a_1, b_1, a_2, b_2\}$ such that $\partial D$ is analytic near $z_0$. 
There exists  $\wt G(\kappa; D; a_1, b_1; a_2, b_2)\in (0,\infty)$ such that as $\eps\to 0$,
\begin{align*}
\Pb\left[\dist(z_0, \eta_i)< \eps, i=1,2 \right] = \wt G(\kappa; D; a_1, b_1; a_2, b_2; z_0) \eps^{\bxi_4} (1+o(1)).
\end{align*}
\end{theorem}

The following lemma states that an SLE$_\kappa$ has a positive probability to stay in a given tube. 
\begin{lemma}\label{lem:sle_tube}
Fix $\kappa\in(0,4]$. Suppose that $\eta$ is an SLE$_\kappa$ in $\Db$ from $a\in \partial \Db\setminus \{1\}$ to $1$. Let $\gamma$ be a simple curve starting from $a$, terminating at $1$, and otherwise not hitting $\partial \Db$. Let $A_\eps$ be the $\eps$ neighborhood of $\gamma$. We have $\Pb[\eta\subset A_\eps]>0$.
\end{lemma}
\begin{proof}
Note that $A:=A_\eps\cap \Db$ is a simply connected domain. Let $f$ be the conformal map from $A$ onto $\Db$ that leaves $a, 1$ fixed. Let $\Pb_A$ be the probability measure of an SLE$_\kappa$ in $A$ from $a$ to $1$.
Let $\alpha(\kappa)$ be given by \eqref{eq:c_kappa}. 
Let $\Lambda_{\Db} (\eta, \Db\setminus A)$ be the mass of Brownian loops in $\Db$ that intersect both $\eta$ and $\Db\setminus A$.
By the restriction property of SLE$_\kappa$, we have
\begin{align}\label{eq:res}
\Pb[\eta\subset A]=\int \one_{\eta \subset A}  (f_A'(a) f_A'(1))^{(6-\kappa)/(2\kappa)} \exp\left(- \alpha(\kappa)\Lambda_{\Db} (\eta, \Db\setminus A)  \right) d\Pb_{A}(\eta).
\end{align}
Note that for $\kappa\in(0,4]$, SLE$_\kappa$ a.s.\ does not touch the boundary of the domain (except at its endpoints), hence $\Lambda_{\Db} (\eta, \Db\setminus A)<\infty$ for $\Pb_A$-a.e.\ $\eta$.
Hence the integrand in \eqref{eq:res} is positive for $\Pb_A$-a.e.\ $\eta$, leading to a positive integral.
\end{proof}

\section{Two-arm events for CLE}\label{sec:2arm}
As a warm-up, we first establish up-to-constants estimates for two-arm events in CLE$_\kappa$. This case, involving only one loop, is much easier. Indeed, in the four-arm case, we will have to handle interactions between different crossing loops in the CLE. We use this simple case to illustrate some ideas which will be used again later on.

\begin{definition}[Two-arm events for CLE]
For $\kappa\in(8/3,4]$, let $\Gamma^+$ (resp.\ $\Gamma$) be a CLE$_\kappa$ in $\Hb$ (resp.\ $\Db$). 
We define the \emph{boundary two-arm event} $\ba_2 (\eps, r)$ (resp.\ \emph{interior two-arm event} $\Ac_2(\eps, r)$)  to be the event that there is at least one loop $\gamma$ in $\Gamma^+$ (resp.\ $\Gamma$) that intersects both $\partial B_\eps$ and $\partial B_r$. 
\end{definition}
The goal of this section is to prove the following proposition.
\begin{proposition}\label{prop:two-arm}
For $\kappa\in(8/3,4]$ and $r\in(0,1)$, we have the following up-to-constants estimates for two-arm events in CLE$_\kappa$, as $\eps\to 0$,
\begin{align*}
\Pb(\ba_2(\eps, r)) \asymp \eps^{\xi^+_{2}(\kappa)},\quad \Pb(\Ac_2(\eps, r)) \asymp \eps^{\xi_{2}(\kappa)},
\end{align*}
where $\bxi_2(\kappa)$ and $\xi_2(\kappa)$ are given by \eqref{eq:arm_exp}, and the implicit constants depend on $r, \kappa$.
\end{proposition}

We will perform a Markovian exploration of CLE$_\kappa$ as in \cite{MR2979861}, in order to relate CLE$_\kappa$ loops to SLE$_\kappa$ curves, so that we can apply  Propositions~\ref{prop:beffara} and~\ref{prop:schramm}. For this purpose, we need to restrict ourselves to some good events. 
On the event $\ba_2 (\eps, r)$ (resp.\ $\Ac_2(\eps, r)$), there is a loop $\gamma$ in $\Gamma^+$ (resp.\ $\Gamma$) that intersects $\partial B_\eps$ and $\partial B_r$.  Let $\Gc^+(\eps, r)$ (resp.\ $\Gc(\eps, r)$) be the event that all the loops in $\Gamma^+\setminus\{\gamma\}$ (resp.\ $\Gamma\setminus\{\gamma\}$) which intersect $B_{2r}$ have diameter less than $r/10$.

\begin{lemma}\label{lem:2arm_good}
For $\kappa\in(8/3,4]$ and $r\in(0,1)$, we have, as $\eps\to 0$,
\begin{align*}
\Pb(\ba_2(\eps, r)) \asymp \Pb(\ba_2(\eps, r) \cap \Gc^+(\eps, r)), \quad \Pb(\Ac_2(\eps, r)) \asymp \Pb(\Ac_2(\eps, r) \cap \Gc(\eps, r)).
\end{align*}
\end{lemma}
The proof relies on the Brownian loop soup (BLS) representation of the CLE. We view the events for CLE$_\kappa$ equivalently as events for the BLS with intensity $\alpha$ in the same domain, coupled with CLE$_\kappa$, where $\alpha$ and $\kappa$ are related by \eqref{eq:c_kappa}.
Throughout the proof, we denote by $\Lc_D$ the Brownian loop soup with intensity $\alpha$ in the domain $D$.
\begin{proof}
	We only give a proof for $\Ac_2(\eps, r)$, since the proof for $\ba_2(\eps, r)$ is very similar. Below, we view $\Ac_2(\eps, r)$ as an event for $\Lc_{\Db}$, i.e., there is at least one outermost cluster of $\Lc_{\Db}$ whose boundary intersects both $\partial B_\eps$ and $\partial B_r$, called \emph{crossing cluster} below.
	Let $\bar\Ac_2(\eps, r)\subseteq \Ac_2(\eps, r)$ be the event that there is only one such crossing cluster. 
	 By the BK inequality (see e.g. \cite{MR1385351} and the references therein), we have 
	\begin{equation}\label{eq:bk1}
		\Pb( \bar\Ac_2(\eps, r) ) \gtrsim \Pb( \Ac_2(\eps, r) ).
	\end{equation}
	More concretely, $\Ac_2(\eps, r)\setminus \bar\Ac_2(\eps, r)$ implies $\Ac_2(\eps, r)\square \Ac_2(\eps, r)$, the disjoint occurrence of $\Ac_2(\eps, r)$ and $\Ac_2(\eps, r)$. Note that there is some constant $c<1$ such that $\Pb( \Ac_2(\eps, r) )\le c$. Hence, by the BK inequality, 
	\[
	\Pb( \Ac_2(\eps, r)\setminus \bar\Ac_2(\eps, r) ) \le \Pb( \Ac_2(\eps, r) )^2 \le c\, \Pb( \Ac_2(\eps, r) ),
	\]
	which implies \eqref{eq:bk1} immediately. 
	
	On the event $\bar\Ac_2(\eps, r)$, there is only one crossing cluster, denoted by $\Cc$. Given  $\Cc$, the loops contained in the complement of the filling of $\Cc$ in $\Db$ are distributed as a Brownian loop soup in that remaining domain, conditioned to have no crossing cluster (recall that for a bounded subset $A$ of $\Cb$, its \emph{filling} is defined as the complement of the unique unbounded connected component of $\Cb \setminus \bar A$). Hence, they are stochastically dominated by $\Lc_{\Db}$, which does not contain cluster of diameter greater than $r/10$ with positive probability. Therefore, we conclude 
	\[
	\Pb( \bar\Ac_2(\eps, r) ) \lesssim \Pb(\bar\Ac_2(\eps, r) \cap \Gc(\eps, r)),
	\]
	which combined with \eqref{eq:bk1} finishes the proof.  
\end{proof}

Let us first deal with the boundary case. For this purpose, we consider the following Markovian exploration of $\Gamma^+$.

\begin{expl}\label{expl1}
We explore $\Gamma^+$ along the arc $\ell(t):=-\exp(-it)r$, see Figure~\ref{fig:2arm_CLE}. We trace every loop in $\Gamma^+$ that $\ell$ encounters in the counterclockwise direction, in the order that $\ell$ encounters them. We stop this exploration the first time that we reach $B_{r/2}$, namely we stop at a time that we are tracing along a loop $\gamma$ that intersects $B_{r/2}$, exactly at the moment that $\gamma$ reaches $\partial B_{r/2}$, so that we have discovered a piece $\wh \gamma$ of $\gamma$. 
If none of the loops in $\Gamma^+$ intersect both $\partial B_r$ and $\partial B_{r/2}$, then we stop this process at the time that we have discovered all the loops in $\Gamma$ that intersect $\partial B_r$.

On the event $E_1$ that there exists a loop in $\Gamma^+$ which intersects both $\partial B_r$ and $\partial B_{r/2}$, we define $\gamma$ and $\wh \gamma$ just as above.
Let $t_1$ be the first time (according to the parametrization of $\ell$) that $\ell$ intersects $\gamma$. Let $a:=\ell(t_1)$, which is one endpoint of $\wh\gamma$. Let $b$ be the other endpoint of $\wh\gamma$.
Let $K(t_1)$ be the union of $\ell((0,t_1))$ together with all the loops in $\Gamma^+$ that $\ell((0,t_1))$ intersects. 
Let $H$ be the connected component containing $0$ of $\Hb\setminus \overline{K(t_1) \cup \wh \gamma}$. Let $f$ be the unique conformal map from $H$ onto $\Hb$ with $f(0)=0$, $f'(0)=1$ and $f(\infty)=\infty$.

Let $\Sigma$ be the $\sigma$-algebra generated by $E_1, \wh\gamma$ and by all the loops in $\Gamma^+$ that $\ell((0,t_1))$ intersects.
Note that $f$ and $H$ are measurable w.r.t.\ $\Sigma$.
Conditionally on $\Sigma$ and on $E_1$, the image of $\gamma\setminus \wh\gamma$ under $f$ is a chordal SLE$_\kappa$ in $\Db$ between $f(a)$ and $f(b)$, which we denote by $\wt \gamma$.
\end{expl}

\begin{figure}[t]
	\centering
	\includegraphics[width=0.85\textwidth]{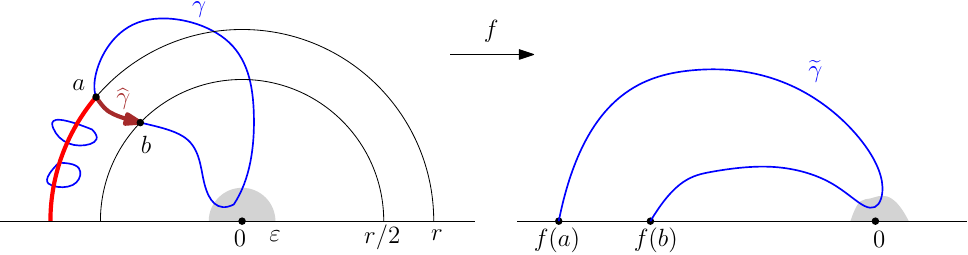}
	\caption{Exploration process \ref{expl1}. The red curve is $\ell([0,t_1])$. The brown curve is $\wh \gamma$.}
	\label{fig:2arm_CLE}
\end{figure}
\begin{proof}[Proof of Proposition~\ref{prop:two-arm}, boundary case]
We perform Exploration process \ref{expl1} for $\Gamma^+$, and use the notations there. We work on the event $E_1$.
 By Lemma~\ref{lem:Koebe}, we have
\begin{align}\label{eq:koebe1}
B_{\eps/4} \cap \Hb \subseteq f(B_\eps \cap \Hb) \subseteq B_{4\eps} \cap \Hb. 
\end{align}
Let $\wt \gamma:=f(\gamma\setminus \wh \gamma)$. 
Let $E_2$ be the event that $\wt\gamma$ intersects $\partial B_{\eps/4}$.
If both $E_1$ and $E_2$ hold, then $\ba_2(\eps, r)$ occurs. Therefore
\begin{align}\label{eq:lo1}
\Pb\big[\ba_2(\eps, r)\big] \ge \Pb[E_1 \cap E_2] =\Eb[ \Pb[E_2 \mid \Sigma] \one_{E_1}].
\end{align}
Conditionally on $\Sigma$, $\wt \gamma$ is a chordal SLE$_\kappa$ in $\Hb$ between $f(a)$ and $f(b)$.
We first apply the conformal map $h: z\mapsto (z-f(b))/(z-f(a))$ to $\wt \gamma$. Since $h(0)=f(b)/f(a)$ and $h'(0)=(f(b)-f(a))/f(a)^2$, we can deduce by Proposition~\ref{prop:schramm} that
\begin{align}\label{eq:lo2}
 \Pb[E_2 \mid \Sigma] \asymp \bigg( \frac{f(b)-f(a)}{f(a) f(b)} \bigg)^{\bxi_2(\kappa)} \eps^{\bxi_2(\kappa)}.
\end{align}
Note that the following quantity does not depend on $\eps$, 
\begin{align}\label{eq:lo3}
\Eb\Bigg[ \bigg( \frac{f(b)-f(a)}{f(a) f(b)} \bigg)^{\bxi_2(\kappa)} \one_{E_1}\Bigg] \in (0, \infty).
\end{align}
The fact that \eqref{eq:lo3} is positive follows from the observations that $E_1$ is an event with positive probability, and that on $E_1$, the quantity $ (f(b)-f(a))/ (f(a)f(b))$ is a.s.\ positive. On the other hand, if \eqref{eq:lo3} was infinite, then by \eqref{eq:lo1} we would have $\Pb\big[\ba_2(\eps, r)\big] =\infty$ for all $\eps$ small enough, which is impossible.
Combining \eqref{eq:lo1}, \eqref{eq:lo2}, \eqref{eq:lo3}, we can deduce the lower bound 
$
\Pb\big[\ba_2(\eps, r)\big]  \gtrsim \eps^{\bxi_2(\kappa)}.
$

Let $E_3$ be the event that $\wt\gamma$ intersects $\partial B_{4\eps}$.
On the event $\ba_2(\eps, r) \cap \Gc^+(\eps, r)$, both $E_1$ and $E_3$ hold. Therefore,
\begin{align*}
\Pb\big[\ba_2(\eps, r) \cap \Gc^+(\eps, r)\big] \le \Pb [E_1 \cap E_3]= \Eb[ \Pb[E_3 \mid \Sigma] \one_{E_1}].
\end{align*}
Similarly to \eqref{eq:lo2}, we have
\begin{align*}
 \Pb[E_3 \mid \Sigma] \asymp \bigg( \frac{f(b)-f(a)}{f(a) f(b)} \bigg)^{\bxi_2(\kappa)} \eps^{\bxi_2(\kappa)}.
\end{align*}
Combined with \eqref{eq:lo3}, we can also get the upper bound
\begin{align}\label{eq:up1}
\Pb\big[\ba_2(\eps, r) \cap \Gc^+(\eps, r)\big]\lesssim \eps^{\xi_2(\kappa)}.
\end{align}
By Lemma~\ref{lem:2arm_good}, this further implies the upper bound $\Pb\big[\ba_2(\eps, r)\big]  \lesssim \eps^{\bxi_2(\kappa)}$, and completes the proof for the boundary case.
\end{proof}

Let us now turn to the interior case. For this purpose, we consider the following Markovian exploration of $\Gamma$.

\begin{expl}\label{expl2}
See Figure~\ref{fig:2arm_CLE_int} for an illustration. Let $\ell$ be a curve which first goes from $-i$ to $-ri$ in a straight vertical line, and then follows the circle $\partial B_r$ in the clockwise direction.
 We trace every loop in $\Gamma$ that $\ell$ encounters in the counterclockwise direction, in the order that $\ell$ encounters them. We stop this exploration the first time that we reach $B_{r/2}$, namely we stop at a time that we are tracing along a loop $\gamma$ that intersects $B_{r/2}$, exactly at the moment that $\gamma$ reaches $\partial B_{r/2}$, so that we have discovered a piece $\wh \gamma$ of $\gamma$. 
If none of the loops in $\Gamma$ intersect both $\partial B_r$ and $\partial B_{r/2}$, then we stop this process at the time that we have discovered all the loops in $\Gamma$ that intersect $[-i, -ri] \cup \partial B_r$.

On the event $E_1$ that there exists a loop in $\Gamma$ which intersects both $\partial B_r$ and $\partial B_{r/2}$, we define $\gamma$ and $\wh \gamma$ just as above.
Let $t_1$ be the first time (according to the parametrization of $\ell$) that $\ell$ intersects $\gamma$. Let $a:=\ell(t_1)$. Let $b$ be the other endpoint of $\wh\gamma$.
Let $K(t_1)$ be the union of $\ell((0,t_1))$ together with all the loops in $\Gamma$ that $\ell((0,t_1))$ intersects. 
Let $U$ be the connected component containing $0$ of $\Db\setminus \overline{K(t_1) \cup \wh \gamma}$. Let $f$ be the unique conformal map from $U$ onto $\Db$ with $f(0)=0$ and $f(a)=i$.

Let $\Sigma$ be the $\sigma$-algebra generated by $E_1$, $\wh\gamma$ and by all the loops in $\Gamma$ that $\ell((0,t_1))$ intersects. Note that $f$ and $U$ are measurable w.r.t.\ $\Sigma$.
Conditionally on $\Sigma$ and on $E_1$, the image of $\gamma\setminus \wh\gamma$ under $f$ is a chordal SLE$_\kappa$ in $\Db$ between $f(a)$ and $f(b)$, which we denote by $\wt \gamma$.
\end{expl}

\begin{figure}[t]
	\centering
	\includegraphics[width=0.75\textwidth]{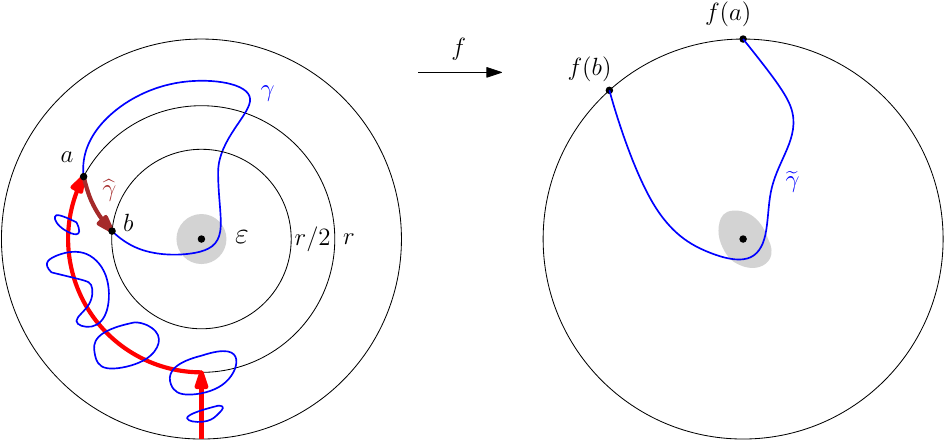}
	\caption{Exploration process \ref{expl2}. The red curve is $\ell([0,t_1])$. The brown curve is $\wh \gamma$.}
	\label{fig:2arm_CLE_int}
\end{figure}
\begin{proof}[Proof of Proposition~\ref{prop:two-arm}, interior case]
We perform Exploration process \ref{expl2} for $\Gamma$, and use the notations there.
We work on the event $E_1$. By the Schwarz lemma, we have $1\le f'(0) \le 2/r$.
By the Koebe $1/4$ theorem, we have
\begin{align}\label{eq:koebe2}
B_{\eps/4} \subseteq f(B_\eps) \subseteq B_{8\eps/r}. 
\end{align}
Let $E_2$ be the event that $\wt \gamma$ intersects $B_{\eps/4}$. If both $E_1$ and $E_2$ hold, then $\Ac_2(\eps, r)$ occurs. Therefore
\begin{align}\label{eq:int_lo1}
\Pb\big[\Ac_2(\eps, r)\big] \ge \Pb[E_1 \cap E_2] =\Eb[ \Pb[E_2 \mid \Sigma] \one_{E_1}].
\end{align}
On $E_1$, conditionally on $\Sigma$, $\wt \gamma$ is a chordal SLE$_\kappa$ in $\Hb$ between $f(a)$ and $f(b)$. By Proposition~\ref{prop:beffara}, we have
\begin{align}\label{eq:int_lo2}
\Pb[E_2 \mid \Sigma] \asymp \eps^{\xi_2(\kappa)} |f(a) - f(b)|^{\bxi_2(\kappa)}.
\end{align}
Note that the following quantity does not depend on $\eps$, 
\begin{align}\label{eq:int_lo3}
\Eb\left[  |f(a) - f(b)|^{\bxi_2(\kappa)} \one_{E_1}\right] \in (0, \infty).
\end{align}
The reason that \eqref{eq:int_lo3} is in $(0,\infty)$ is similar to \eqref{eq:lo3}. Combining \eqref{eq:int_lo1}, \eqref{eq:int_lo2}, \eqref{eq:int_lo3}, we can deduce the lower bound 
$\Pb\big[\Ac_2(\eps, r)\big]  \gtrsim \eps^{\xi_2(\kappa)}.$

Let $E_3$ be the event that $\wt\gamma$ intersects $\partial B_{8\eps/r}$. Similarly to \eqref{eq:up1}, we can deduce that 
\begin{align*}
\Pb\big[\Ac_2(\eps, r) \cap \Gc(\eps, r)\big] \le \Pb [E_1 \cap E_3]= \Eb[ \Pb[E_3 \mid \Sigma] \one_{E_1}]\lesssim \eps^{\xi_2(\kappa)}.
\end{align*}
By Lemma~\ref{lem:2arm_good}, this further implies the upper bound $\Pb\big[\Ac_2(\eps, r)\big]  \lesssim \eps^{\xi_2(\kappa)}$, and completes the proof for the interior case.
\end{proof}

\section{Separation lemma for CLE} \label{sec:separation}

This section is dedicated to deriving separation lemmas for CLE, or equivalently loop-soup clusters, in various settings. They are in fact slightly stronger than the simplified version stated in the introduction as Theorem~\ref{thm:sep_cle}.

This section is structured as follows. We first consider the interior case in Section~\ref{subsec:i_sep}, where we give a stronger separation lemma (Proposition~\ref{prop:sep_i}), that will eventually be used in the derivation of Thereom~\ref{thm:main}, and as an intermediate step, we introduce an alternative version for the BLS (Lemma~\ref{lem:sl}), which is stated in terms of a key notion called quality. We collect results for the BLS in Section~\ref{subsec:summary}, which are analogs in the continuum of properties derived in \cite{GNQ1} for the discrete setting, i.e.\ the random walk loop soup. Next, we use these results to establish Lemma~\ref{lem:sl} and Proposition~\ref{prop:sep_i} in Section~\ref{subsec:proof_i_sep}. We then show the up-to-constants estimates between $\ione(\eps,r)$ and $\itwo(\eps,r)$ in Section~\ref{subsec:equivalence}. Finally, in Section~\ref{subsec:boundary}, we summarize the corresponding results in the boundary case.

\subsection{Interior case}\label{subsec:i_sep}
Suppose $r<1/2$. On the event $\itwo(\eps,r)$, there are two loops $\gamma_1$ and $\gamma_2$ that intersect both 
$\partial B_\eps$ and $\partial B_r$.  Let $\gitwo(\eps, r)$ be the event that the following conditions all hold
\begin{itemize}
\item $\gamma_j \subset (B_r \cup B_{r/10}(re^{i(j-1)\pi})) \setminus B_{r/10}(re^{i j \pi})$ for $j=1,2$,
\item all the loops in $\Gamma\setminus\{\gamma_1, \gamma_2\}$ have diameter less than $r/40$.
\end{itemize}

\begin{proposition}\label{prop:sep_i}
For $\kappa\in(8/3,4]$, there exists a constant $c>0$ such that for all $0<2\eps< r<1/2$,
\begin{align*}
\Pb(\itwo(\eps,r))\le c\, \Pb(\itwo(\eps, r) \cap \gitwo(\eps, r)).
\end{align*}
\end{proposition}

\begin{remark}
		This type of separation lemma also holds in the discrete setting, if we replace the Brownian loop soup clusters by the random walk loop soup clusters everywhere in the definition of the events.	
\end{remark}
	
Below, we work with the BLS. Recall that we use $\Lc_\Db$ to denote the BLS in $\Db$ with intensity $\alpha$, and let $\eps\le s\le r$. We write $\Lc_s$ for the set of loops in $\Lc_\Db$ that are contained in $B_s$. We first introduce some notions for deterministic objects.
For any loop configuration $L_s$ in $B_s$, we can define the quality associated with $L_s$ as follows, where for any set $A\subset \Cb$ and any loop configuration $L$, we use $\Lambda(A,L)$ to denote the union of $A$ and all the clusters in $L$ that intersect $A$. We say that $(\eta_1,\eta_2)$ is an admissible pair of excursions in $B_s$ if $\Lambda(\eta_1,L_s)$ and $\Lambda(\eta_2,L_s)$ both intersect $B_\eps$, but they do not intersect each other. If we denote the starting and ending points of $\eta_i$ by $x_i$ and $y_i$, respectively, we say that the admissible pair $(\eta_1,\eta_2)$ is $\delta$-separated (at scale $s$) if 
\begin{equation}\label{eq:poe}
	\Lambda(\eta_1\cup B_{\delta s}(x_1)\cup B_{\delta s}(y_1),L_s)\cap (\eta_2\cup B_{\delta s}(x_2)\cup B_{\delta s}(y_2))=\emptyset.
\end{equation}
Furthermore, we say that $(\gamma_1,\gamma_2)$ is a $\delta$-good pair of loops (at scale $s$) if each $\gamma_i$ contains an excursion $\eta_i$ such that $(\eta_1,\eta_2)$ is a $\delta$-separated admissible pair, and $\Lambda(\gamma_1,L_s)\cap \Lambda(\gamma_2,L_s)=\emptyset$. Note that if $(\gamma_1,\gamma_2)$ is a $\delta$-good pair of loops, then it is $\delta'$-good for all $\delta' \in (0,\delta)$. 

Given the loop configuration $L_s$ in $B_s$, we consider the \emph{extension probability} across $s$ (from $\eps$) by $\delta$-good pairs of loops in $\Lc_{2s}\setminus \Lc_s$ (i.e., the loop soup at the next scale), which is defined as
\begin{equation}\label{eq:m_s}
	m_s(\delta; \eps, L_s)=\Pb\left( \begin{split}
		\text{there is a $\delta$-good pair of loops $(\gamma_1,\gamma_2)$ in $\Lc_{2s}\setminus \Lc_s$ such that the two clusters in}\\ 
			\text{$L_s\cup (\Lc_{2s}\setminus \Lc_s)$ containing $\gamma_1$ and $\gamma_2$, respectively, are disjoint and outermost}
	\end{split}  \right).
\end{equation}
Since $m_s(\delta; \eps, L_s)$ is decreasing in $\delta$, the following \emph{quality} at scale $s$ is well-defined:
\begin{equation}\label{eq:quality}
	Q_\eps(L_s):=\sup\{ \delta\in [0,1/4]: m_s(\delta; \eps, L_s)\ge \delta \}.
\end{equation}
Then, our separation lemma for loop-soup clusters can be stated as follows, in terms of the above-mentioned quality.
\begin{lemma}\label{lem:sl}
	For all $\alpha\in (0,1/2]$, there exist constants $u,c>0$ such that for all $0<2\eps< r<1/2$,
	\begin{equation}\label{eq:sep-1}
		\Pb( Q_\eps(\Lc_r)\ge u )\ge c\, \Pb( \itwo(\eps,r) ).
	\end{equation}
\end{lemma}

\subsection{Summary of corresponding results in \cite{GNQ1} for BLS}\label{subsec:summary}
We now state continuous versions of several key results that were derived in \cite{GNQ1} in the discrete setting, including various separation lemmas, a locality property, and a quasi-multiplicativity upper bound. Since they can all be established in a similar way as in \cite{GNQ1}, we refer the reader to that paper and skip the proofs. These properties will be used in the derivation of the separation lemma for CLE.

\paragraph{Separation lemma.}
We start with a version of separation lemma for Brownian motions inside a BLS. The setup is analogous to that of Proposition~4.7 in \cite{GNQ1}. 
Let $L_\eps$ be a loop configuration in $B_\eps$.
Let $V_1$ and $V_2$ be two disjoint subsets of $B_\eps$, which both intersect $\partial B_\eps$. Let $j,k\ge 1$. Let $\bar x=(x_1,\ldots,x_j)$ be a vector of $j$ vertices in $V_1\cap\partial B_\eps$ and $\bar y=(y_1,\ldots,y_k)$ be a vector of $k$ vertices in $V_2\cap \partial B_\eps$ (some of the points in $\bar x$ may coincide, and similarly with $\bar y$).
We view the quintuple $(L_\eps,V_1,V_2,\bar x,\bar y)$ as an initial configuration, and we restrict to the case when $\Lambda(V_1,L_\eps)\cap V_2=\emptyset$.
For $\eps\le s\le r$, let $\Pi^1_s$ (resp.\ $\Pi^2_s$) be the union of $j$ (resp.\ $k$) independent Brownian motions started, respectively, from each of the $j$ points in $\bar x$ (resp.\ each of the $k$ points in $\bar y$), and stopped upon reaching $\partial B_s$. We require that all of these Brownian motions are independent. 
Finally, let $D\supseteq\Db$, and let $\Lc_{D}$ be the BLS in $D$ with intensity $\alpha$, which is 
independent of all the previous Brownian motions.
We use $\Lc_{\eps,D} := \Lc_D \setminus \Lc_\eps$ to denote the loop soup made of the loops in $\Lc_{D}$ which are not entirely contained in $B_\eps$. 
The quality at $s$ is then defined as
\begin{align} \notag
	Q^{j,k}(s) := \sup & \Big\{ \delta \ge 0 \: : \: \Lambda \Big( V_1 \cup \Pi^1_s \cup \Big( \cup_{z \in \Pi^1_s \cap \partial B_s} B_{\delta s}(z) \Big), L_\eps \uplus \Lc_{\eps,D} \Big)\\
	& \hspace{4cm} \cap \Big( V_2 \cup \Pi^2_s \cup \Big( \cup_{z \in \Pi^2_s \cap \partial B_s} B_{\delta s}(z) \Big) \Big) = \emptyset \Big\}. \label{eq:Qjk}
\end{align}
In this definition, we consider the unions of, respectively, $j$ and $k$ balls, all with radius $\delta s$, centered on the hitting points along $\partial B_s$ of each of the $j$ Brownian motions in $\Pi^1_s$, and of each of the $k$ Brownian motions in $\Pi^2_s$.
The following separation result for two packets of Brownian motions is analogous to Proposition~4.7 in \cite{GNQ1}.

\begin{lemma}\label{lem:sep-jk}
	For all $j,k\ge 1$, there exists a constant $c(j,k)>0$ such that the following holds. For all $0<2\eps< r<1/2$ and $D \supseteq \Db$, for each initial configuration $(L_\eps,V_1,V_2,\bar x,\bar y)$ with $\bar x=(x_1,\ldots,x_j)$ in $V_1\cap \partial B_\eps$ and $\bar y=(y_1,\ldots,y_k)$ in $V_2\cap \partial B_\eps$, and for any intensity $\alpha\in (0,1/2]$ of the BLS under consideration,
	\begin{equation} \label{eq:sep_twopackets}
		\Pb \big( Q^{j,k}(r)>1/(10(j+k)) \mid Q^{j,k}(r)>0 \big) \ge c.
	\end{equation}
	Moreover, \eqref{eq:sep_twopackets} also holds with $Q^{j,k}(r)$ replaced by $\bar Q^{j,k}(r):=Q^{j,k}(r)\ind_\Dc$, for any event $\Dc$ which is decreasing for the BLS.
\end{lemma}

\begin{remark}\label{rmk:general-sl}
In \cite{GNQ1}, we also derive separation lemmas in other settings, which all remain valid here in the continuous setting. In particular, we make use later of a separation lemma in the reversed direction, and one concerning non-disconnection events (corresponding to Propositions 4.8 and 4.9 of \cite{GNQ1}, respectively). For brevity we do not repeat their statements (nor their proofs), since these are exact analogs of their discrete counterparts.
\end{remark}

\paragraph{Locality and the quasi-multiplicativity.}
With Lemma~\ref{lem:sep-jk} at hand (and variations, as explained in Remark~\ref{rmk:general-sl} above), we can derive the locality and quasi-multiplicativity properties for the continuous arm events, as we did for their discrete counterparts in \cite{GNQ1}.
We first extend the definition of of arm events to a general domain. For $D\supseteq B_{2r}$, we let $\itwo(\eps,r;D)$ be the event that there are two outermost clusters in $\Lc_D$ across the annulus $A_{\eps,r}$.
Define the truncated arm event with respect to the loop soup $\Lc_D$ by
\[
\overrightarrow\itwo(\eps,r;D) := \itwo(\eps,r;D) \cap \{ \Lambda(B_\eps, \Lc_D) \subseteq B_{2r} \}.
\]

\begin{lemma}[Locality]\label{lem:locality}
	For all $0<2\eps< r<1/2$, $D\supseteq B_{2r}$ and $\alpha\in (0,1/2]$,
	\[
	\Pb( \itwo(\eps,r;D) ) \lesssim \Pb( \overrightarrow\itwo(\eps,r;D) ).
	\]
\end{lemma}

\begin{lemma}[Quasi-multiplicativity]
	For any $\alpha\in (0,1/2]$, there exists a constant $c(\alpha)>0$ such that for all $0 < r_1\le r_2/2 \le r_3/16$ and $D \supseteq B_{2r_3}$,
	\begin{align}\label{eq:2n-quasi-1}
		\Pb(  \itwo(r_1,r_3;D) ) \le c\, \Pb( \itwo(r_1,r_2;B_{2r_2}) )\, \Pb( \itwo(4r_2,r_3;D) ).
	\end{align}
\end{lemma}

\subsection{Proof of the separation lemma for CLE}\label{subsec:proof_i_sep}
In what follows, we aim to prove Lemma~\ref{lem:sl} concerning the quality $Q_\eps(\Lc_s)$, and then use it to prove the main separation result, i.e.\ Proposition~\ref{prop:sep_i}. We begin by establishing some key ingredients, and we then complete the proof of separation in the end.

Let $M:=\lfloor\log_2(r/\eps)\rfloor$ ($\geq 1$ by assumption). Set $r_0:=\eps$. For all $1\le i\le M$, let $r_i:=2^i\eps$, and let $\Ac(i):= \itwo(\eps,r_i;B_{r_{i+1}})$ be the \emph{local arm event}. Let $\wt\Ac(i)\subseteq \Ac(i)$ be the \emph{stable arm event} that there are exactly two clusters $\Cc_1$ and $\Cc_2$ in $\Lc_{r_{i+1}}$ across the annulus $A_{\eps,r_i}$, and for each $j=1,2$,
\begin{equation}\label{eq:cond}
	\text{$\Cc_j$ contains a subcluster of $\Lc_{r_{i}}$ that crosses $A_{\eps,r_{i-1}}$.}
\end{equation}

The following lemma upper bounds $\Pb(\Ac(i))$ by a combination of $\Pb( \wt\Ac(i) )$ and $\Pb( \Ac(i-1) )$, which can be roughly regarded as a kind of stability result on arm events. This stability result is an additional technicality in the proof of Lemma~\ref{lem:sl}, compared to those we proved in \cite{GNQ1}. 
It turns out to be quite important to make the recursive procedure \eqref{eq:QML1} work (which is a standard step in the proof of separation lemma). 
It is worth putting this lemma in the first place in order to emphasize its importance. The other two ingredients Lemma~\ref{lem:sep-2} and Lemma~\ref{lem:sep-1} are more standard (although additional complexities emerge due to the convoluted definition of quality), which are indeed analogous to Lemma~4.5 and Lemma~4.4 in \cite{GNQ1}, respectively. 

\begin{figure}[t]
	\centering
	\includegraphics[width=0.56\textwidth]{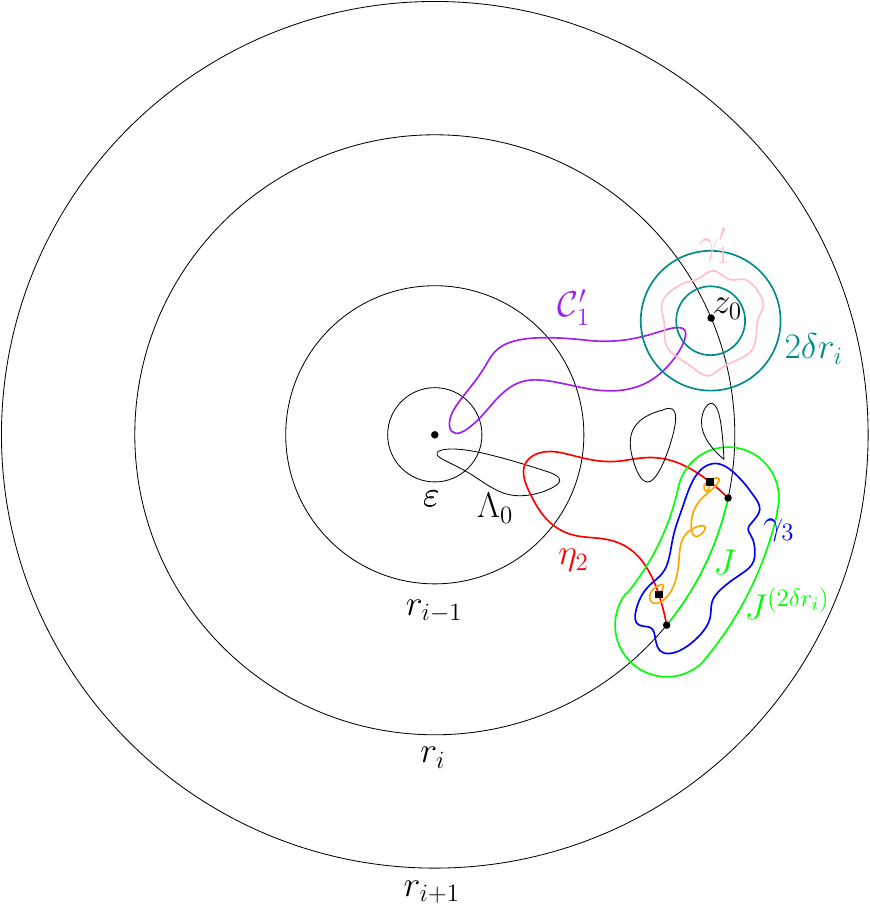}
	\caption{Case (2a) in the proof of Lemma~\ref{lem:quasi}. The subcluster $\Cc_1'$ is in purple and it intersects some ball $B_{\delta r_i}(z_0)$ with $z_0\in \partial B_{r_i}$ (small one in cyan). The larger ball (in cyan) centered at $z_0$ is of radius $2\delta r_i$. $\gamma_1'$ in pink is a loop inside $A_{\delta r_i,1.5\delta r_i}(z_0)$ that surrounds $B_{\delta r_i}(z_0)$. $\eta_2$ in red is an excursion part of the loop $\gamma_2$ in $B_{r_i}$ that intersects $\Lambda_0$. The distinguished arc $J$ that joins the endpoints of $\eta_2$ is in green, and its $(2\delta r_i)$-sausage is enclosed by the green loop. $\gamma_3$ is a loop in $J^{(2\delta r_i)}\setminus J^{(\delta r_i)}$ that surrounds $J^{(\delta r_i)}$. The two small squares represent the first hitting of $J^{(\delta r_i)}$ by $\eta_2$ from some point on $\eta_2\cap\Lambda_0$ in both directions. These two squares are connected by an orange bridge inside $J^{(2\delta r_i)}\cap B_{r_i}$, which is used to construct the new loop $\gamma_2'$. We also draw some clusters in black hooked by $\eta_2\cup J^{(2\delta r_i)}$ in $\Lc_{r_i}$.}
	\label{fig:quasi}
\end{figure}

\begin{lemma}\label{lem:quasi}
	There is a constant $c>0$ such that the following holds. 
	For all $\delta>0$, there exists a constant $b(\delta)>0$ such that for all $i\ge 2$, 
	\begin{equation}\label{eq:quasi}
		\Pb( \Ac(i) ) \le b(\delta)\, \Pb( \wt\Ac(i) ) + \delta^c\, \Pb( \Ac(i-1) ).
	\end{equation}
\end{lemma}
\begin{proof}
	Let $\bar\Ac(i)\subseteq \Ac(i)$ be the
	event that there are exactly two crossing clusters of $A_{\eps,r_i}$ in $\Lc_{r_{i+1}}$. By the BK inequality (similar to \eqref{eq:bk1}), we have $\Pb( \Ac(i) ) \lesssim \Pb( \bar\Ac(i) ) $. Hence, it suffices to show \eqref{eq:quasi} for $\bar\Ac(i)$ in place of $\Ac(i)$.
	First, we explore all the loops in $\Lc_{r_i}$, and let $\Lambda_0:=\Lambda(B_\eps, \Lc_{r_i})$. On the event $\bar\Ac(i)$, we can find two loops $\gamma_1$ and $\gamma_2$ in $\Lc_{r_{i+1}}$ that both intersect $\Lambda_0$ and $\partial B_{r_i}$, and they belong to two different crossing clusters in $\Lc_{r_{i+1}}$, which are denoted by $\Cc_1$ and $\Cc_2$ below. For each $j=1,2$, we check if $\Cc_j$ satisfies the condition \eqref{eq:cond}. 
	
	\textbf{Case (1)}.
	We first consider the case that both $\Cc_1$ and $\Cc_2$ do not satisfy \eqref{eq:cond}. 
	Given $\Cc_1$ and $\Cc_2$, all the other clusters in $\Lc_{r_{i+1}}$ that intersect $B_\eps$ are contained in $B_{r_{i-1}}$ with some positive probability, which is assumed now.  
	Then, we have $\Lambda_0\subset B_{r_{i-1}}$, and $\gamma_1$ and $\gamma_2$ are a pair of non-intersecting loops that intersect $\Lambda_0$ and cross $A_{r_{i-1},r_i}$. 
	In fact, this is just a continuous analogue of the discrete event $\bar \Ec_2$, which we already dealt with in Lemma~5.8 of \cite{GNQ1}.  
	Using a similar Markovian exploration therein with  Lemma~\ref{lem:sep-jk} as an input now, we can show the result as follows (we refer the reader to the proof of Lemma~5.8 of \cite{GNQ1} for details). With a positive cost, one can modify the loops $\gamma_1$ and $\gamma_2$ such that they are contained in $B_{r_i}$, well-separated at scale $1.5\,r_{i-1}$, and they retain the connection (i.e., the refreshed crossing clusters in $\Lc_{r_i}$ containing them are disjoint and cross $A_{\eps,r_{i-1}}$). After this modification, we can extend the two crossing clusters to scale $r_i$ by using two loops in $\Lc_{r_{i+1}}$ that intersect $\partial B_{r_i}$. This implies $\wt\Ac(i)$, and it only costs a positive constant because the refreshed crossing clusters are well-separated now. We conclude $\Pb( \bar\Ac(i) )\lesssim \Pb( \wt\Ac(i) )$ in this case.
	
	\textbf{Case (2)}.
	Next, we consider the case that \eqref{eq:cond} is true for $\Cc_1$ but not for $\Cc_2$. We can further assume that there is a subcluster $\Cc'_1$ of $\Lc_{r_i}$ that crosses $A_{\eps,r_{i-1}}$ with $\dist(\Cc'_1,\partial B_{r_i})\le \delta r_i$, and hooked by $\gamma_1$. Otherwise, we are in a similar situation as case (1): two loops crossing $A_{(1-\delta)r_i,r_i}$ that make clusters in $B_{(1-\delta)r_i}$ connect to $\partial B_{r_i}$. Note that the multiplicative constant now depends on $\delta$, which is allowed. Next, we choose some $z_0\in \partial B_{r_i}$ such that $B_{\delta r_i}(z_0) \cap \Cc'_1\neq\emptyset$. Let $\eta_2$ (chosen in arbitrary way) be an excursion of $\gamma_2$ in $B_{r_i}$ such that it intersect $\Lambda_0$. Let $J$ be the arc on $\partial B_{r_i}$ that joins the endpoints of $\eta_2$ and $J\cup \eta_2$ does not surround $\Cc'_1$. Let $J^{(a)}:=\{w: \dist(w,J)\le a\}$ be the \emph{$a$-sausage} of $J$. We will consider the following two subcases:
	\begin{itemize}
		\item \textbf{Case (2a)}: There does \emph{not} exist a ball $B_{2\delta r_i}(z)$ with $z\in \partial B_{r_i}$ such that $B_{2\delta r_i}(z)\cap\Cc_1'\neq\emptyset$ and $(\Cc_1'\cup B_{2\delta r_i}(z))\cap \Lambda(\eta_2\cup J^{(2\delta r_i)}, \Lc_{r_i})\neq\emptyset$. Since $B_{\delta r_i}(z_0) \cap \Cc'_1\neq\emptyset$ by our choice of $z_0$, it follows that $(\Cc_1'\cup B_{2\delta r_i}(z_0))\cap \Lambda(\eta_2\cup J^{(2\delta r_i)}, \Lc_{r_i})=\emptyset$. See Figure~\ref{fig:quasi} for an illustration. Therefore, with some positive cost depending only on $\delta$, we can make all the following happen: 
		\begin{itemize}
			\item There is a loop $\gamma_1'$ in $A_{\delta r_i,1.5\delta r_i}(z_0):=B_{1.5\delta r_i}(z_0)\setminus B_{\delta r_i}(z_0)$ that surrounds $B_{\delta r_i}(z_0)$.
			\item There is a loop $\gamma_3$ in $J^{(2\delta r_i)}\setminus J^{(\delta r_i)}$ that surrounds $J^{(\delta r_i)}$.
			\item We truncate $\eta_2$ when it first hits $J^{(\delta r_i)}$ from some point on $\eta_2\cap\Lambda_0\subset B_{r_{i-1}}$ in both directions, and then hook the two hitting points inside $J^{(2\delta r_i)}\cap B_{r_i}$ by a bridge to get a new loop $\gamma'_2$. Note that $\gamma'_2\subset B_{r_i}$ makes $\Ac(i-1)$ occur, and $\gamma'_2\cap\gamma_3\neq\emptyset$ makes $\Ac(i)$ occur.
		\end{itemize}
		By requiring all the other clusters made by the remaining unexplored loops in $\Lc_{r_{i+1}}$ to have diameter smaller than $\delta r_i/100$ (which occurs with positive probability depending only on $\delta$), we conclude $\Pb( \bar\Ac(i) ) \le b(\delta)\, \Pb( \wt\Ac(i) )$ in this case. 
		\item \textbf{Case (2b)}: There is a ball $B_{2\delta r_i}(z)$ with $z\in \partial B_{r_i}$ such that $B_{2\delta r_i}(z)\cap\Cc_1'\neq\emptyset$ and $(\Cc_1'\cup B_{2\delta r_i}(z))\cap \Lambda(\eta_2\cup J^{(2\delta r_i)}, \Lc_{r_i})\neq\emptyset$. 
		In this case, there exists a point $w\in J$, measurable with respect to $\Lc_{r_i}\cup\{\gamma_2\}$, such that the following event occurs
		\[
		H(z,w):=\{\text{either $B_{2\delta r_i}(z)$ or $B_{2\delta r_i}(w)$ is not surrounded by any loop in $\Lc_{r_{i+1}}$}\}.
		\]
		To be more precise, we will choose the previous possible $z,w$ to be the ones that are closest to the rightmost point $(r_i,0)$. 
		On the one hand, we have the uniform estimate that $\Pb(H(z,w))\le \delta^c$ for some constant $c>0$ (by considering surrounding loops inside dyadic annuli). On the other hand, we note that $\Lc_{r_i}\cup\{\gamma_2\}$ already makes the occurrence of two crossing clusters of $A_{\eps,r_{i-1}}$, and hence, by slightly adapting the proof of Lemma~\ref{lem:locality}, we can resample the loops in $\Lc_{r_i}\cup\{\gamma_2\}$ with some positive cost such that $\Ac(i-1)$ happens. Combined, we obtain an upper bound $\delta^c\, \Pb( \Ac(i-1) )$ in this case.
	\end{itemize}
	Combining the estimates in all the above cases, we conclude the proof of Lemma~\ref{lem:quasi}. 
\end{proof}

Recall that the extension probability $m_s(\delta;\eps,L_s)$ and the quality $Q_\eps(L_s)$ were defined in \eqref{eq:m_s} and \eqref{eq:quality}, respectively.
In what follows, $\eps$ is considered as fixed, so we drop it from the notation, and we write $m_s(\delta):=m_s(\delta; \eps, \Lc_s)$ and  $\Qc(s):=Q_\eps(\Lc_s)$ for any $\eps\le s\le r$, for simplicity. Note that both $m_s(\delta)$ and $\Qc(s)$ are measurable with respect to the loop soup $\Lc_s$ in $B_s$.

\begin{lemma}\label{lem:sep-2}
	For all $\theta>0$, there exists $\rho(\theta)>0$ such that for all $i$,
	\[
	\Pb( 0<\Qc(r_i)< \rho, \wt\Ac(i) ) \le \theta\, \Pb( \Ac(i-1) ).
	\]
\end{lemma}
\begin{proof}
	We first sample $\Lc_{r_i}$ such that both $\Ac(i-1)$ and $0<\Qc(r_i)<\rho$ occur. To realize $\wt\Ac(i)$, there exists a $\delta$-good pair of loops $(\gamma_1,\gamma_2)$ inside the collection $\Lc_{r_{i+1}}\setminus \Lc_{r_i}$ for some $\delta>0$, which extend two clusters in $\Lc_{r_i}$ across $A_{\eps,r_{i-1}}$ to $\partial B_{r_i}$ disjointly. 
	We first assume that $(\gamma_1,\gamma_2)$ is not $\rho$-good at scale $r_i$, that is, \eqref{eq:poe} fails to hold for any pair of excursions $(\eta_1,\eta_2)$ in $B_{r_i}$ extracted from $(\gamma_1,\gamma_2)$. As we already argued in case (2b) of the proof of Lemma~\ref{lem:quasi}, the probability that $B_{\rho r_i}(z)$ with $z\in \partial B_{r_i}$ is not surrounded by a loop in $\Lc_{r_{i+1}}$ is bounded by $\rho^c$. Hence, the probability that the previous event holds for all endpoints $z$ of $\eta_1$ and $\eta_2$ (which is implied by $\wt\Ac(i)$) is uniformly bounded by $4\rho^c$, which gives the desired upper bound in this case.   
	On the other hand, since $\Qc(r_i)<\rho$, we have $m_{r_i}(\rho)<\rho$ by \eqref{eq:quality}. Therefore, the probability of existence of a $\rho$-good pair of loops is smaller than $\rho$. This finishes the proof, by picking $\theta=4\rho^c+\rho$.
\end{proof}

\begin{lemma}\label{lem:sep-1}
	There is a constant $u>0$ such that the following holds. For any $\rho>0$, there is a constant $v(\rho)>0$ such that for all $i$,
	\[
	\Pb( \Qc(r_{i+1})\ge u ) \ge v(\rho)\, \Pb( \Qc(r_i)\ge \rho ).
	\]
\end{lemma}

\begin{figure}[t]
	\centering
	\subfigure{\includegraphics[width=0.67\textwidth]{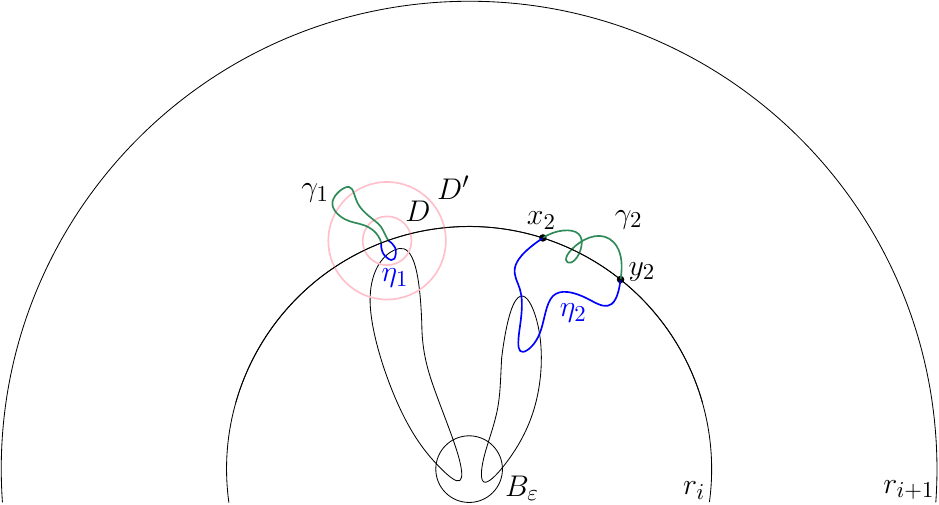}}
	\subfigure{\includegraphics[width=0.67\textwidth]{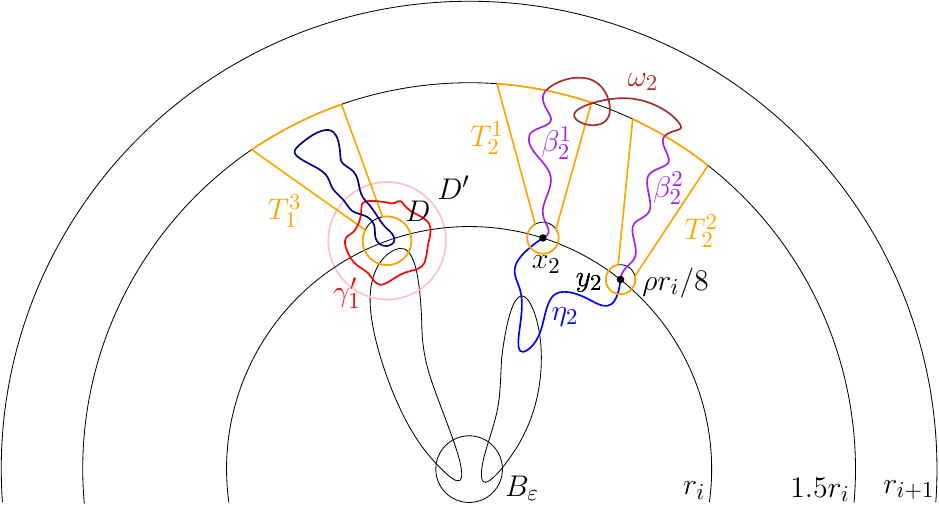}}
	\caption{Proof of Lemma~\ref{lem:sep-1}. \emph{Top:} Illustration of the case $\gamma_1$ in case (1) and $\gamma_2$ in case (2). The $\rho$-separated admissible pair of excursions $(\eta_1,\eta_2)$ is in blue. $D$ and $D'$ are the smaller ball and the larger ball in pink respectively. \emph{Bottom:} The configuration after resampling. $\gamma'_1$ is in red, $\beta_2^1$ and $\beta_2^2$ are in purple, and $\omega_2$ is in brown. The cones with the associated balls are in orange.}
	\label{fig:sep}
\end{figure}

\begin{proof}
	Consider $\Lc_{r_i}$ such that $\Qc(r_i)\ge \rho$, that is, $m_{r_i}(\rho)\ge \rho$ by \eqref{eq:quality}. Hence, with probability at least $\rho$, there exists a $\rho$-good pair of loops $(\gamma_1,\gamma_2)$ in $\Lc_{r_{i+1}}\setminus \Lc_{r_i}$ at scale $r_i$, which are such that the clusters in $\Lc_{r_{i+1}}$ containing them are disjoint and outermost. We now assume that this event occurs, and extract from $(\gamma_1,\gamma_2)$ a $\rho$-separated admissible pair of excursions in $B_{r_i}$. We select such a pair of excursions in an arbitrary way, and denote it by $(\eta_1,\eta_2)$ from now on. Assume $\eta_j$ is from $x_j$ to $y_j$. Now, the ending parts of $\eta_1$ and $\eta_2$ are $(2\rho r_i)$-away in $\Lc_{r_i}$, in the sense of \eqref{eq:poe} with $\delta s$ replaced by $\rho r_i$. We cover $\partial B_{r_i}$ by a chain of $O(\rho^{-1})$ balls of radius $\rho r_i/4$, with centers on $\partial B_{r_i}$, such that the centers of two neighboring balls are at distance $\rho r_i/4$ (with possibly at most one exception, for which the distance is smaller than $\rho r_i/4$). For each $j=1,2$, we distinguish the following two cases.
	
	\textbf{Case (1)}. Suppose $\eta_j\subset D$ for some ball $D$ in the previous chain. Let $D'$ be the ball of radius $\rho r_i/2$ concentric with $D$. We further consider two subcases. If $\gamma_j\subset D'$, then we add a loop that stays in $D'\setminus D$ and surrounds $D$ with some universal positive probability. If $\gamma_j$ is not contained in $D'$, then $\gamma_j$ crosses the annulus $D'\setminus D$. By the scale invariance of the Brownian loop measure, the total mass of $\gamma_j$ such that $\gamma_j$ crosses $D'\setminus D$ and $\gamma_j\subset B_{2r_i}$ is bounded by some constant $c_1(\rho)$. Moreover, the total mass of $\gamma_j$ such that $\gamma_j\cap B_{2r_i}^c\neq\emptyset$ and $\gamma_j$ does not surround $B_{r_i}$ is bounded by some constant $c_2$. Therefore, the total mass of $\gamma_j$ that does not stay in $D'$ is bounded by $c_1(\rho)+c_2$. On the other hand, the total mass of loops in $D'\setminus D$ that surround $D$ is bounded from below by some constant $c_3$. Combining these two estimates, it follows that we can resample $\gamma_j$ by a new loop $\gamma'_j$ that stays in $D'\setminus D$ and surrounds $D$ with a cost at least $c_3/(c_1(\rho)+c_2)$. We refer to Figure~\ref{fig:sep} for an illustration ($j=1$).
	
	\textbf{Case (2)}. Suppose $|x_j-y_j|\ge \rho r_i/4$. Then, the total mass of all bridges in $B_{r_{i+1}}$ connecting $x_j$ and $y_j$ is bounded by the Green's function $G_{B_{r_{i+1}}}(x_j,y_j)$, which is $=O(|\log\rho|)+c_4$ for some constant $c_4$.
	We now construct two cones $T^1_j$ and $T^2_j$ such that
	\begin{itemize}
		\item $T^1_j\cap B_{r_i}\subset B_{\rho r_i/8}(x_j)$ and $T^2_j\cap B_{r_i}\subset B_{\rho r_i/8}(y_j)$. Both of the arcs $T^1_j\cap \partial B_{\rho r_i/8}(x_j)$ and $T^2_j\cap \partial B_{\rho r_i/8}(y_j)$ have length greater than $C_1\rho r_i$ for some small constant $C_1$.
		\item The sides of each cone are truncated when they intersect $\partial B_{1.5 r_i}$, and the arc on $\partial B_{1.5 r_i}$ that joins the endpoints consists of the remaining side of the cone.
		\item Both cones do not intersect $B_{\rho r_i}(x_{3-j})\cup B_{\rho r_i}(y_{3-j})$.
	\end{itemize} 
	Given the excursion $\eta_j$, we can reconstruct the remaining bridge in the following way with a constant cost $c_5(\rho)$:
	\begin{itemize}
		\item Let $\beta_j^1$ (resp.\ $\beta_j^2$) be a Brownian motion started from $x_j$ (resp.\ $y_j$) and restricted to exit $T_j^1\cup B_{\rho r_i/8}(x_j)$ (resp.\ $T_j^2\cup B_{\rho r_i/8}(y_j)$) from $\partial B_{1.5 r_i}$ and stopped there.
		\item Let $\omega_j$ be a Brownian path from the endpoint of $\beta_j^2$ to that of $\beta_j^1$, such that it is contained in a $(C_2 r_i)$-sausage of the arc between the endpoints (for some $C_2 > 0$ small enough).
		\item We construct the new loop $\gamma'_j$ by the concatenation $\gamma'_j:=\eta_j\oplus\beta_j^2\oplus\omega_j\oplus [\beta_j^1]^R$. Note that this decomposition is unique since $\eta_j$ is the unique admissible excursion of $\gamma'_j$ that has diameter greater than $\rho r_i/4$.
	\end{itemize}
	It follows that we can resample $\gamma_j$ by a new loop $\gamma'_j$ constructed above with cost at least $c_6(\rho)$. See Figure~\ref{fig:sep} for $j=2$.
	
	In fact, for case (1), we can also construct a suitable cone $T^3_j$ with a similar form such that $T^3_j\cap B_{r_i}\subset D$. Then, with probability at least $c_7(\rho)>0$, there exists a loop (in navy in Figure~\ref{fig:sep}) in $T^3_j$ that intersects  $D$ and surrounds an arc on $\partial B_{1.4 r_i}$ of length greater than $C_3 r_i$. 
	
	Since there are two cases for each $j$, there are four cases in total. We will only consider the case $j=1$ in case (1) and $j=2$ in case (2) for an illustration. We do the resampling for both values of $j$ as above, to get the desired new loops, which has some constant cost $c_8(\rho):=c_3\, c_6(\rho)\, c_7(\rho)/(c_1(\rho)+c_2)$. We can further choose the cones such that $T^3_1$ and $T^1_2\cup T^2_2$ are well-separated from each other, in the sense that if we enlarge the angle of each cone by some fixed constant proportion, then they are still disjoint. Moreover, all the clusters that intersect each cone can be forced to stay in the respective larger cone with some positive probability $c_9(\rho)$. We further require all the other clusters in $\Lc_{r_{i+1}}$ that intersect $\partial B_{r_i}$, except the two containing $\gamma_1$ and $\gamma_2$ respectively, to have diameter smaller than $r_i/100$, which occurs with probability $c_{10}>0$. Resampling the loops $\gamma_j$ if necessary and assuming all the above events occur, we see that $m_{r_{i+1}}(u_1)\ge u_2$ for some constants $u_1,u_2>0$. Therefore, by the FKG inequality (\cite[Theorem 20.4]{LP2017}), we obtain that 
	\[
	\Pb( m_{r_{i+1}}(u_1)\ge u_2 ) \ge \rho \,c_8(\rho)\, c_9(\rho) \, c_{10}\, \Pb( \Qc(r_i)\ge \rho ),
	\]
	which implies Lemma~\ref{lem:sep-1} immediately, with $u = u_1 \wedge u_2$.
\end{proof}

\begin{proof}[Proof of Lemma~\ref{lem:sl}]
	Let $u$ be the constant from Lemma~\ref{lem:sep-1}. By Lemma~\ref{lem:locality}, we have $\Pb(\itwo(\eps,r))\lesssim \Pb( \Ac(M) )$. 
	Hence, we only need to show that
	\begin{equation}\label{eq:AML}
		\Pb( \Ac(M) ) \lesssim \Pb( \Qc(r_M)\ge u ),
	\end{equation}
	which combined with the fact that $\Pb( \Qc(r_M)\ge u ) \lesssim \Pb( \Qc(r)\ge u )$ (since $\frac{r}{2} < r_M \leq r$) would finish the proof of  Lemma~\ref{lem:sl}.
	
	Let $G$ be the event that all clusters in $\Lc_\eps$ have diameter smaller than $\eps/100$, which occurs with probability at least $c_1$. 
	By definition, if we consider the loop configuration $\Lc_\eps$ on the event $G$, we have $m_\eps(1/4)\ge a_1$ for some constant $a_1>0$, which implies $\Qc(\eps) \ge a_2:=1/4\wedge a_1$. Using Lemma~\ref{lem:sep-1} repeatedly, we obtain that 
	\begin{equation}\label{eq:QML3}
		\Pb( \Qc(r_M)\ge u ) \ge v(a_2)\, v(u)^{M-1}\, \Pb( G ) \ge c_1\, v(a_2)\, v(u)^{M-1}. 
	\end{equation}
	For all $i\ge 2$, by Lemma~\ref{lem:sep-2}, for any $\theta>0$, there exists $\rho(\theta)>0$ such that
	\begin{align*}
		\Pb( \wt\Ac(i) ) &\le \Pb( \Qc(r_i)\ge \rho ) + \Pb( 0<\Qc(r_i)<\rho, \wt\Ac(i) ) \\
		& \le \Pb( \Qc(r_i)\ge \rho ) + \theta\, \Pb( \Ac(i-1) ),
	\end{align*}
	which combined with Lemma~\ref{lem:quasi} implies that 
	\[
	\Pb( \Ac(i) ) \le b(\delta)\, \Pb( \Qc(r_i)\ge \rho ) + (b(\delta)\, \theta+\delta^c) \Pb( \Ac(i-1) ).
	\]
	Iterating the above inequality, we obtain
	\begin{equation}\label{eq:QML1}
		\Pb( \Ac(M) ) \le \sum_{i=0}^{M-2} b(\delta)\, (b(\delta)\,\theta+\delta^c)^i\, \Pb( \Qc(r_{M-i})\ge \rho ) + (b(\delta)\,\theta+\delta^c)^{M-1}.
	\end{equation}
	Applying Lemma~\ref{lem:sep-1} again, we have
	\begin{align}\notag
		\Pb( \Qc(r_{M-i})\ge \rho ) &\le v(\rho)^{-1} \,\Pb( \Qc(r_{M-i+1})\ge u ) \\ \label{eq:QML2}
		&\le v(\rho)^{-1}\, v(u)^{-i+1} \, \Pb( \Qc(r_{M})\ge u ).
	\end{align}
	Plugging \eqref{eq:QML3} and \eqref{eq:QML2} into \eqref{eq:QML1}, 
	\[
	\Pb( \Ac(M) ) \le \Pb( \Qc(r_M)\ge u ) \, \left( \frac{b(\delta)\,v(u)}{v(\rho)}\, \sum_{i=0}^{M-2} \left(\frac{b(\delta)\,\theta+\delta^c}{v(u)}\right)^i + \frac{(b(\delta)\,\theta+\delta^c)^{M-1}}{c_1\, v(a_2)\, v(u)^{M-1}} \right).
	\]
	We first choose $\delta$ sufficiently small to satisfy $\delta^c\le v(u)/4$, and for this $\delta$ we take $\theta$ to be small enough such that $b(\delta)\,\theta\le v(u)/4$. Hence, $\frac{b(\delta)\,\theta+\delta^c}{v(u)}\le 1/2$, and \eqref{eq:AML} follows. This completes the proof of Lemma~\ref{lem:sl}.
\end{proof}

\begin{proof}[Proof of Proposition~\ref{prop:sep_i}]
	By Lemma~\ref{lem:sl} and noting that $\itwo(\eps, r)\subset \itwo(\eps, r/2)$, we have $\Pb( \Qc(r/2)\ge u )\ge c\, \Pb( \itwo(\eps, r) )$. Hence, it suffices to show that $\Pb( \itwo(\eps, r) \cap \gitwo(\eps, r) )\ge c\, \Pb( \Qc(r/2)\ge u )$. But this follows immediately by using a similar construction of loops as in the proof of Lemma~\ref{lem:sep-1}, combined with a similar argument as in the proof of Lemma~\ref{lem:2arm_good} to control the size of all other clusters. We omit the details for brevity. 
\end{proof}

\subsection{Equivalence between different four-arm events}\label{subsec:equivalence}

We start by showing the following up-to-constants estimate.
\begin{lemma}\label{lem:compareE12}
	For $\kappa\in(8/3,4]$, we have 
	\begin{equation*}
		\Pb(\ione(\eps, r)) \lesssim \Pb(\itwo(\eps, r)) \lesssim \Pb(\ione(2\eps, r)).
	\end{equation*}
\end{lemma}

\begin{proof}
	We first show that $\Pb(\itwo(\eps, r)) \lesssim \Pb(\ione(2\eps, r))$.
	Let $\Xi$ be the set of loops in $\Lc_{\Db}$ that stay in $A_{\eps,2\eps}$ and surround $B_\eps$. Since the measure of such loops is bounded from below uniformly in $\eps$, there is a universal constant $c>0$ such that 
	\begin{equation}\label{eq:Xi}
		\Pb( \Xi=\emptyset ) \le c\, \Pb( \Xi\neq\emptyset ).
	\end{equation}
	Note that $\itwo(\eps, r)$ is equivalent to $\itwo(\eps, r)(\Lc_{\Db}\setminus\Xi) \cap \{ \Xi=\emptyset \}$, where the notation $\itwo(\eps, r)(\Lc_{\Db}\setminus\Xi)$ means that $\Lc_{\Db}\setminus\Xi$ satisfies $\itwo(\eps, r)$. By the independence of $\Xi$ and $\Lc_{\Db}\setminus\Xi$, we obtain that
	\begin{equation}\label{eq:E2Xi}
		\Pb( \itwo(\eps, r) ) =\Pb( \itwo(\eps, r)(\Lc_{\Db}\setminus\Xi) \cap \{ \Xi=\emptyset \} ) =  \Pb( \itwo(\eps, r)(\Lc_{\Db}\setminus\Xi)) \, \Pb( \Xi=\emptyset ). 
	\end{equation}
	Combining \eqref{eq:Xi} with \eqref{eq:E2Xi}, and noting that $\itwo(\eps, r)(\Lc_{\Db}\setminus\Xi)\cap \{ \Xi\neq\emptyset \}$ implies $\ione(2\eps,r)$, we have 
	\begin{equation*}
		\Pb( \itwo(\eps, r) ) \le c\, \Pb( \itwo(\eps, r)(\Lc_{\Db}\setminus\Xi)) \, \Pb( \Xi\neq\emptyset ) \le c\, \Pb( \ione(2\eps,r) ).
	\end{equation*}
	
	Next, we show the harder direction $\Pb(\ione(\eps, r)) \lesssim \Pb(\itwo(\eps, r))$. The general proof idea is very similar to the proof of locality for arm events. Hence, we will only sketch the key idea, and refer to Section~5.2 in \cite{GNQ1} for the corresponding details in the discrete setting. Note that there are different cases for $\ione(\eps, r)$, as shown in Figure~\ref{fig:4arm_CLE}. But all of them can be tackled in a very similar way. To fix idea, we will focus on the middle case with C-shape, i.e., there are two pairs of crossings of the loop $\gamma$ in $A_{\eps,r}$ that are not connected by the intersection of $\gamma$ with $B_r$. We denote this event by $\ionec(\eps, r)$. On the event $\ionec(\eps, r)$, we first explore all the loops in $B_r$ to check if there are clusters made by these loops that cross $A_{\eps,r/2}$. If there are at least two of them, then we end up with $\itwo(\eps, r/2;B_r)$. By Proposition~\ref{prop:sep_i}, it is easy to see that $\Pb(\itwo(\eps, r/2;B_r))\lesssim\Pb(\itwo(\eps, r))$.
	Hence, we only need to consider the following two cases, belonging to $\ionec(\eps, r)\setminus \itwo(\eps, r/2;B_r)$. 
	\begin{itemize}
		\item \textbf{Case (1)}. No cluster of $\Lc_r$ across $A_{\eps,r/2}$. There are two different ways to realize the event $\ionec(\eps, r)$. First, assume there are two Brownian loops $\gamma_1$ and $\gamma_2$ crossing $A_{r/2,r}$ such that each $\gamma_i$ contains an excursion $\eta_i$ inside $B_r$ that are connected to two different clusters $\Cc_1$ and $\Cc_2$ inside $B_{r/2}$, respectively, and  $\Lambda(\eta_1\cup\Cc_1,\Lc_r)$ does not intersect $\eta_2\cup\Cc_2$. Note that we allow $\gamma_1$ and $\gamma_2$ to intersect away from $\eta_1$ and $\eta_2$. Now, we view each excursion $\eta_i$ as a pair of Brownian paths from $r/2$ to $r$: by  Lemma~\ref{lem:sep-jk} with $j=k=2$, we see that the ending parts of these excursions are away from each other inside $\Lc_r$. Hence, we can resample the remaining parts of $\gamma_i\setminus\eta_i$ for both $i=1,2$ such that they stay inside local regions around the ending points of $\eta_i$, respectively, and the loops $\gamma_1$ and $\gamma_2$ after resampling are also away from each other inside $\Lc_r$. Finally, by resampling all the remaining loops in $\Lc_{\Db}\setminus(\Lc_r\cup\{\gamma_1,\gamma_2\})$ such that the clusters formed by them are sufficiently small, we end up with $\itwo(\eps, r)$. Since every resampling procedure only costs some positive constant (independent of $\eps$), we conclude this case. Next, we assume that there is a Brownian loop $\wt\gamma$ that contains two disjoint excursions $\wt\eta_1$ and $\wt\eta_2$, which are connected to two clusters $\Cc_1$ and $\Cc_2$ inside $B_{r/2}$ as before, and satisfy the same condition, where the only difference is that both excursions come from the same loop. Then, we can do the same surgery as before to split the large loop $\wt\gamma$ into two separated small loops, that contain $\wt\eta_1$ and $\wt\eta_2$, respectively, such that $\itwo(\eps, r)$ occurs in the end (note that in this subcase, the cost can depend on $\alpha$ or $\kappa$, but it is still independent of $\eps$). This finishes the proof in this case. 

		\item \textbf{Case (2)}. Exactly one cluster of $\Lc_r$ across $A_{\eps,r/2}$. If there already exists a cluster $\Cc$ of $\Lc_r$ that crosses $A_{\eps,r/2}$, then we can reduce it to the case that $\Cc$ is not too big, that is, $\Cc\subset B_{3r/4}$ (similar to the argument used in the proof of Lemma~5.9 in \cite{GNQ1}). On the event $\ionec(\eps, r)$, there is a loop $\wh\gamma$ that contains an excursion $\wh\eta$ inside $B_r$, which is connected to some other cluster $\Cc'$ inside $B_{r/2}$, and furthermore, the filling of  $\Lambda(\wh\eta\cup\Cc',\Lc_r)$ does not intersect $\Cc$. Still, we allow the remaining part $\wh\gamma\setminus\wh\eta$ to intersect $\Cc$. By a continuous version of separation lemma associated with disconnection events (see Remark~\ref{rmk:general-sl}), analogous to Proposition 4.9 in \cite{GNQ1}, we get that $\wh\eta$ is separated from $\Cc$ at scale $r$. Hence, we can reconnect the endpoints of $\wh\eta$ to get a new loop, which is also separated from $\Cc$. This implies $\itwo(\eps, r)$.
	\end{itemize}
	
	This completes the proof for $\ionec(\eps, r)$. The other possible case, indicated by the left part of Figure~\ref{fig:4arm_CLE}, can be proved in a similar way by using an inward exploration and a continuous version of the reversed separation lemma (corresponding to Proposition 4.8 in \cite{GNQ1}), so we omit the details. 
\end{proof}

\subsection{Boundary case}\label{subsec:boundary}
Now, we give a corresponding separation lemma for $\Gamma^+$.
Suppose $0<2\eps< r<1/2$. On the event $\btwo(\eps,r)$, there are two loops $\gamma_1$ and $\gamma_2$ that intersect both 
$\partial B_\eps$ and $\partial B_r$.  Let $\gbtwo(\eps, r)$ be the event that the following conditions all hold
\begin{itemize}
	\item $\gamma_1 \subset B_r \cup B_{r/10}(re^{\pi i/4}) \setminus B_{r/10}(re^{3\pi i/4})$ and $\gamma_2 \subset B_r \cup B_{r/10}(re^{3\pi i/4}) \setminus B_{r/10}(re^{\pi i/4})$,
	\item all the loops in $\Gamma^+\setminus\{\gamma_1, \gamma_2\}$ that intersect $\Db$ have diameter less than $r/40$.
\end{itemize}

\begin{lemma}\label{lem:sep_b}
	For $\kappa\in(8/3,4]$, there exists a constant $c>0$ such that for all $0<2\eps< r<1/2$,
	\begin{align*}
		\Pb(\btwo(\eps,r))\le c\, \Pb(\btwo(\eps, r) \cap \gbtwo(\eps, r)).
	\end{align*}
\end{lemma}

The proof is similar to that of Proposition~\ref{prop:sep_i}, and thus omitted.
Using Lemma~\ref{lem:sep_b}, and adapting the proof of Lemma~\ref{lem:compareE12}, it is not hard to derive the following result.

\begin{lemma}\label{lem:compareE12_b}
	For $\kappa\in(8/3,4]$, we have 
	\begin{equation*}
		\Pb(\bone(\eps, r)) \asymp \Pb(\btwo(\eps, r)).
	\end{equation*}
\end{lemma}
\begin{proof}
	We only give a sketch of proof for the direction
	\[
	\Pb(\btwo(\eps, r)) \lesssim \Pb(\bone(\eps, r)).
	\]
	First, by Lemma~\ref{lem:sep_b}, we can replace the left-hand side by the event $\btwo(\eps, r) \cap \gbtwo(\eps, r)$, which is assumed to hold now. On this event, we can find a Brownian loop in $B_{r/10}(re^{\pi i/4})\cup B_{r/10}(re^{3\pi i/4})\cup B_r^c$ that surrounds both $B_{r/20}(re^{\pi i/4})$ and $B_{r/20}(re^{3\pi i/4})$, to realize the event $\bone(\eps, r)$. Note that the total mass of such Brownian loops is positive and only depends on $r$. Thus, we conclude the result immediately.  
\end{proof}

\section{Proof of Theorem~\ref{thm:main}} \label{sec:est_CLE}
In this section, we prove Theorem~\ref{thm:main}. We deal with the interior case in Section~\ref{sec:int_cle} and the boundary case in Section~\ref{sec:bdy_cle}.

\subsection{Interior four-arm event for CLE}\label{sec:int_cle}

\begin{figure}
	\centering
	\includegraphics[width=0.8\textwidth]{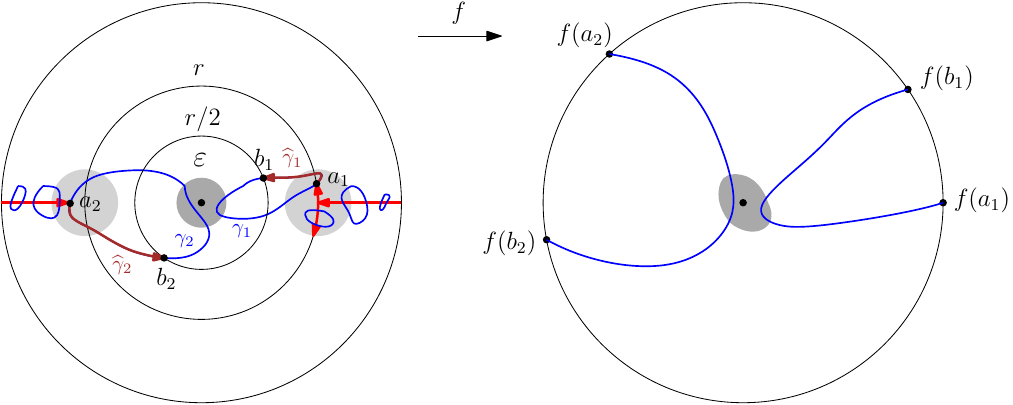}
	\caption{Exploration process performed on a configuration of $\Gamma$ for which $\itwo(\eps, r) \cap \gitwo(\eps, r)$ occurs. The red curves represent $\ell_i([0,t_i])$ for $i=1,2$.}
	\label{fig:4arm_CLE_int}
\end{figure}

\begin{proof}[Proof of \eqref{eq:i4arm} in Theorem~\ref{thm:main}]
By Lemma~\ref{lem:compareE12}, it suffices to show that 
\begin{equation}\label{eq:pe2}
	\Pb(\itwo(\eps, r)) \asymp \eps^{\xi_{4}(\kappa)}.
\end{equation}
Let us first perform the following exploration process, see Figure~\ref{fig:4arm_CLE_int}. For $i=1,2$, let $\ell_{i,1}$ be the horizontal line which goes from $e^{(i-1)\pi}$ to $re^{(i-1)\pi}$. Let $\ell_{i,2}$ (resp.\ $\ell_{i,3}$) be the clockwise (resp.\ counterclockwise) arc on $\partial B_r$ from $re^{(i-1)\pi}$ to the first point that it intersects $\partial B_{r/10}(r)$.
Let $\ell_i$ be the concatenation of the three curves $\ell_{i,1}$, $\ell_{i,2}$ and $\ell_{i,3}$, in this order. Note that $\ell_i$ is not a curve, but a succession of three curves. 
 We trace every loop in $\Gamma$ that $\ell_i$ encounters in the counterclockwise direction, in the order that $\ell_i$ encounters them. 
 We stop this exploration the first time that we reach $B_{r/2}$, namely we stop at a time that we are tracing along a loop $\gamma_i$ that intersects $B_{r/2}$, exactly at the moment that $\gamma_i$ reaches $\partial B_{r/2}$, so that we have discovered a piece $\wh \gamma_i$ of $\gamma_i$. 
If none of the loops in $\Gamma$ intersect both $\ell_i$ and $\partial B_{r/2}$, then we stop this process at the time that we have discovered all the loops in $\Gamma$ that intersect $\ell_i$.

On the event $E_i$ that there exists a loop in $\Gamma$ which intersects both $\ell_i$ and $\partial B_{r/2}$, we define $\gamma_i$ and $\wh \gamma_i$ just as above. 
Let $a_i$ and $b_i$ be the endpoints of $\wh \gamma_i$ (traced from $a_i$ to $b_i$ in the counterclockwise direction).
Let $t_i$ be the first time (according to the parametrization of $\ell_i$) that $\ell_i$ intersects $\gamma_i$. Let $K_i$ be the union of $\wh\gamma_i$, $\ell_i((0,t_i))$ together with all the loops in $\Gamma$ that $\ell_i((0,t_i))$ intersects. 
Let $E_3$ be the event that $E_1 \cap E_2$ occurs. 
On the event $E_3$, let $U$ be the connected component containing $0$ of $\Db \setminus \overline{K_1 \cup K_2}$. Let $f$ be the unique conformal map from $U$ onto $\Db$ with $f(0)=0$ and $f(a_1)=1$.

Let $\Sigma$ be the $\sigma$-algebra generated by $E_3$, $\wh\gamma_1, \wh\gamma_2$ and by all the loops in $\Gamma$ that $\ell_i((0,t_i))$ intersects. Note that $f$ and $U$ are measurable w.r.t.\ $\Sigma$.
Conditionally on $\Sigma$ and on $E_3$, the image of $(\gamma_1 \cup \gamma_2) \setminus (\wh\gamma_1\cup \wh\gamma_2)$ under $f$ is a pair of chordal $2$-SLE$_\kappa$ in $\Db$ with the endpoints $f(b_1), f(a_1), f(b_2), f(a_2)$, that we denote by $(\wt \gamma_1, \wt \gamma_2)$. Note that the pairing pattern of the $2$-SLE$_\kappa$ is not measurable w.r.t.\ $\Sigma$. Conditionally on $\Sigma$ and on $E_3$, the probability that $f(a_1)$ is connected to $f(b_1)$ is given by an explicit function of the cross-ratio of the four points $f(b_1), f(a_1), f(b_2), f(a_2)$, which was computed in \cite{MR3843419}. In particular, if $f(a_1)$ is connected to $f(b_2)$, then $\gamma_1$ and $\gamma_2$ are in fact the same loop, but this information (whether $\gamma_1$ and $\gamma_2$ are the same) is not measurable w.r.t.\ $\Sigma$.

On the event $E_3$, by the Schwarz lemma, we have $1\le f'(0)\le 2/r$. By the Koebe $1/4$ theorem, we have
\begin{align*}
	B_{\eps/4} \subseteq f(B_\eps) \subseteq B_{8\eps/r}.
\end{align*}
Let $E_4$ be the event that $\wt \gamma_1$ connects $f(a_1), f(b_1)$. Let $E_5$ be the event that both $\wt \gamma_1$ and $\wt\gamma_2$ intersect $B_{\eps/4}$.
If $E_3, E_4$ and $E_5$ all hold, then $\itwo(\eps, r)$ occurs. Therefore
\begin{align}\label{eq:e2_lo}
\Pb[ \itwo(\eps, r)] \ge \Pb[E_3 \cap E_4 \cap E_5] =\Eb \big[\Pb[E_4 \cap E_5 \mid \Sigma] \one_{E_3} \big].
\end{align}
On $E_3$, conditionally on $\Sigma$, let $p(f(b_1), f(a_1), f(b_2), f(a_2))$ be the probability that $E_4$ occurs. If we further condition on $E_4$, then the probability that $E_5$ occurs is given by Theorem~\ref{thm:zhan}. 
It follows that there exists a function $F$ depending on $\kappa, f(b_1), f(a_1), f(b_2), f(a_2)$ such that as $\eps\to 0$,
\begin{align}\label{eq:e45}
\Pb[E_4 \cap E_5 \mid \Sigma] \one_{E_3}\asymp F(\kappa, f(b_1), f(a_1), f(b_2), f(a_2))  \eps^{\xi_4}  \one_{E_3}.
\end{align}
Note that the following quantity does not depend on $\eps$, 
\begin{align}\label{eq:F}
\Eb\left[F(\kappa, f(b_1), f(a_1), f(b_2), f(a_2)) \one_{E_3} \right] \in (0,\infty).
\end{align}
The fact that \eqref{eq:F} is positive is obvious, since $E_3$ is an event with positive probability, and on $E_3$, the quantity $F(\kappa, f(b_1), f(a_1), f(b_2), f(a_2))$ is a.s.\ positive. On the other hand, if \eqref{eq:F} was infinite, then by \eqref{eq:e2_lo} we would have $\Pb[ \itwo(\eps, r)] =\infty$ for all $\eps$ small enough, which is impossible.
Combining \eqref{eq:e2_lo}, \eqref{eq:e45} and \eqref{eq:F}, we can deduce the lower bound 
$$\Pb[ \itwo(\eps, r)]  \gtrsim \eps^{\xi_4}.$$

Let $E_6$ be the event that both $\wt \gamma_1$ and $\wt\gamma_2$ intersect $B_{8\eps/r}$.
On the event $\itwo(\eps,r) \cap \gitwo(\eps,r)$, the events $E_3, E_4$ and $E_6$ all occur. Therefore
\begin{align*}
\Pb[\itwo(\eps,r) \cap \gitwo(\eps,r)] \le \Pb[E_3\cap E_4 \cap E_6]=\Eb \big[\Pb[E_4 \cap E_6 \mid \Sigma] \one_{E_3} \big].
\end{align*}
Similarly to \eqref{eq:e45}, we have
\begin{align*}
\Pb[E_4 \cap E_5 \mid \Sigma] \one_{E_3}\asymp F(\kappa, f(b_1), f(a_1), f(b_2), f(a_2))  \eps^{\xi_4}  \one_{E_3}.
\end{align*}
Combined with \eqref{eq:F}, we can also get the upper bound 
\begin{align*}
\Pb[\itwo(\eps,r) \cap \gitwo(\eps,r)] \lesssim \eps^{\xi_4}.
\end{align*}
By Proposition~\ref{prop:sep_i}, we then deduce the upper bound 
$$
\Pb[\itwo(\eps,r)] \lesssim \eps^{\xi_4}.
$$
This completes the proof of \eqref{eq:pe2}, which implies  \eqref{eq:i4arm} in Theorem~\ref{thm:main}.
\end{proof}

\subsection{Boundary four-arm event}\label{sec:bdy_cle}

\begin{figure}[t]
	\centering
	\includegraphics[width=0.9\textwidth]{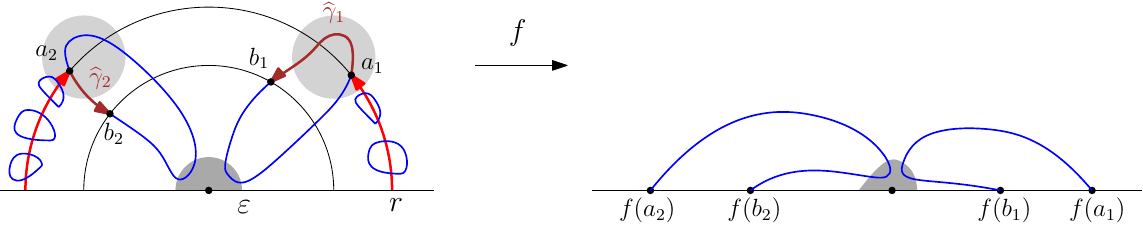}
	\caption{Exploration process performed on a configuration of $\Gamma^+$ for which $\btwo(\eps, r) \cap \gbtwo(\eps, r)$ occurs. The red curves represent $\ell_i([0,t_i])$ for $i=1,2$.}
	\label{fig:4arm_CLE_bdy}
\end{figure}

\begin{proof}[Proof of \eqref{eq:b4arm} in Theorem~\ref{thm:main}]
By Lemma~\ref{lem:compareE12_b}, it suffices to show that 
	\begin{equation}\label{eq:btwo}
		\Pb[\btwo(\eps, r)] \asymp \eps^{\xi^+_{4}(\kappa)}.
\end{equation}
We explore $\Gamma^+$ along the arcs $\ell_1:=r \exp (it)$ for $t\in(0, 3\pi/8)$ and $\ell_2:=-r\exp(-it)$ for $t\in(0, 3\pi/8)$, see Figure~\ref{fig:4arm_CLE_bdy}. 
For $i=1,2$, we trace every loop in $\Gamma$ that $\ell_i$ encounters in the counterclockwise direction, in the order that $\ell_i$ encounters them. 
 We stop this exploration the first time that we reach $B_{r/2}$, namely we stop at a time that we are tracing along a loop $\gamma_i$ that intersects $B_{r/2}$, exactly at the moment that $\gamma_i$ reaches $\partial B_{r/2}$, so that we have discovered a piece $\wh \gamma_i$ of $\gamma_i$. 
If none of the loops in $\Gamma^+$ intersect both $\ell_i$ and $\partial B_{r/2}$, then we stop this process at the time that we have discovered all the loops in $\Gamma^+$ that intersect $\ell_i$.

On the event $E_i$ that there exists a loop in $\Gamma^+$ which intersects both $\ell_i$ and $\partial B_{r/2}$, we define $\gamma_i$ and $\wh \gamma_i$ just as above. 
Let $a_i$ and $b_i$ be the endpoints of $\wh \gamma_i$ (traced from $a_i$ to $b_i$ in the counterclockwise direction).
Let $t_i$ be the first time (according to the parametrization of $\ell_i$) that $\ell_i$ intersects $\gamma_i$. Let $K_i$ be the union of $\wh\gamma_i$, $\ell_i((0,t_i))$ together with all the loops in $\Gamma^+$ that $\ell_i((0,t_i))$ intersects. 
Let $E_3$ be the event that $E_1 \cap E_2$ occurs. 
On the event $E_3$, let $H$ be the connected component containing $0$ of $\Hb \setminus \overline{K_1 \cup K_2}$. Let $f$ be the unique conformal map from $H$ onto $\Hb$ with $f(0)=0$, $f(\infty)=\infty$ and $f'(0)=1$.

Let $\Sigma$ be the $\sigma$-algebra generated by $E_3$, $\wh\gamma_1, \wh\gamma_2$ and by all the loops in $\Gamma^+$ that $\ell_i((0,t_i))$ intersects. Note that $f$ and $H$ are measurable w.r.t.\ $\Sigma$.
Conditionally on $\Sigma$ and on $E_3$, the image of $(\gamma_1 \cup \gamma_2) \setminus (\wh\gamma_1\cup \wh\gamma_2)$ under $f$ is a pair of chordal $2$-SLE$_\kappa$ in $\Hb$ with the endpoints $f(b_1), f(a_1), f(b_2), f(a_2)$, that we denote by $(\wt \gamma_1, \wt \gamma_2)$. Note that the pairing pattern of the $2$-SLE$_\kappa$ is not measurable w.r.t.\ $\Sigma$. Conditionally on $\Sigma$ and on $E_3$, the probability that $f(a_1)$ is connected to $f(b_1)$ is given by an explicit function of the cross-ratio of the four points $f(b_1), f(a_1), f(b_2), f(a_2)$, which was computed in \cite{MR3843419}.

On the event $E_3$, by Lemma~\ref{lem:Koebe}, we have
\begin{align*}
B_{\eps/4}\cap \Hb \subseteq f(B_\eps\cap \Hb) \subseteq B_{4\eps}\cap \Hb.
\end{align*}
The remainder of the proof is almost the same as that of the interior case in Section~\ref{sec:int_cle}, except that we will use Theorem~\ref{thm:zhanb}, instead of Theorem~\ref{thm:zhan}.
We can similarly get the upper bound
\begin{align*}
\Pb[\btwo(\eps,r) \cap \gbtwo(\eps,r)] \lesssim \eps^{\bxi_4},
\end{align*}
which then implies \eqref{eq:btwo} by Lemma~\ref{lem:sep_b}. This completes the proof.
\end{proof}

\section{A general version of four-arm event for SLE} \label{sec:est_SLE}
In this section, we aim to prove Theorem~\ref{thm:sle_int_utc}, which provides up-to-constants estimates for general versions of interior and boundary four-arm events for SLE, defined as $\Wc^+_4 (\eps,r)$ and $\Wc_4 (a, \eps, r)$ in \eqref{eq:def_4arm_b} and \eqref{eq:def_4arm_int}.
We first review in Section~\ref{subsec:comparison} some related results on SLE arm exponents obtained in \cite{MR3846840}. Then, we deal with the interior case in Section~\ref{subsec:proof_thm}, and the boundary case in Section~\ref{subsec:proof_thm_bdy}.

\subsection{Some results on SLE arm exponents in \cite{MR3846840}}\label{subsec:comparison}
We now recall and discuss some related results on SLE arm exponents in \cite{MR3846840}.
We follow the original notations and statements in \cite{MR3846840}, with a few exceptions that we indicate in footnotes. 
Throughout, we write $\xi_{2j}$ for the interior $2j$-arm exponent, and $\bxi_j$ for the boundary $j$-arm exponent, given by \eqref{eq:arm_exp}.\footnote{The definition of the boundary $j$-arm exponent $\alpha_j^+$ in \cite{MR3846840} is shifted by one, namely $\alpha_j^+=\xi_{j+1}^+$.}

First recall the following \emph{boundary arm events} defined in \cite{MR3846840}.
Let $\eta$ be a chordal SLE$_\kappa$ in $\Hb$ from $0$ to $\infty$. Fix $y \le -4r <0 <\eps \le x$. 
Let $\wh\tau_0=\wh\sigma_0=0$. For $j\ge 1$, let $\wh\tau_j$ be the first time after $\wh\sigma_{j-1}$ that $\eta$ hits the connected component of $\partial B_\eps(x) \setminus \eta([0, \wh\sigma_{j-1}])$ that contains $x+\eps$, and let $\wh\sigma_j$ be the first time after $\wh\tau_j$ that $\eta$ hits the connected component of $\partial B_r(y)\setminus \eta([0,\wh\tau_j])$ that contains $y-r$. Define 
\begin{align}\label{eq:h1}
\Hc_{2j-1}^\alpha(\eps, x, y, r):=\{\wh\tau_j<\infty\}.
\end{align}
Let $\wt \tau_0=\wt\sigma_0=0$. For $j\ge 1$, let $\wt\sigma_j$ be the first time after $\wt \tau_{j-1}$ that $\eta$ hits the connected component of $\partial B_r(y) \setminus \eta([0,\wt\sigma_j])$ that contains $y-r$. Let $\wt\tau_j$ be the first time after $\wt\sigma_j$ that $\eta$ hits the connected component of $\partial B_\eps(x) \setminus \eta([0, \wt\sigma_j])$ containing $x+\eps$.
Define
\begin{align}\label{eq:h2}
\Hc_{2j}^\alpha(\eps, x,y,r):=\{\wt\tau_j <\infty\}.
\end{align}
Note that the definitions of $\Hc_{2j-1}^\alpha(\eps, x, y, r)$ and $\Hc_{2j}^\alpha(\eps, x,y,r)$ are restrictive on the arcs that each crossing first hits. For instance, the events depicted in Figure~\ref{fig:4arm_bdy} do not belong to this type of arm events.

The following result is part of \cite[Proposition 3.1]{MR3846840}.
\begin{proposition}[Proposition 3.1, \cite{MR3846840}]
\label{prop:wu_bdy}
Fix $\kappa\in(0,4]$.\footnote{There is another parameter $\rho$ in \cite[Proposition 3.1]{MR3846840}, which corresponds to a force point $v\ge 0$ of the SLE$_\kappa(\rho)$. Here we take $\rho=0$ and $v=x$. The statement in \cite[Proposition 3.1]{MR3846840} was made only for $\kappa\in(0,4)$, but we believe that the same argument also works for $\kappa=4$ in the case considered here}. Suppose $r\ge 1 \wedge (200 \eps)$. We have
\begin{align}
&\Pb[\Hc^\alpha_{2j-1} (\eps, x,y, r)] \lesssim x^{\bxi_{2j-1} -\bxi_{2j}} \eps^{\bxi_{2j}}, \text{ provided }  |y|\ge (40)^{2j-1}r,\\
\label{eq:2j_up}
&\Pb[\Hc^\alpha_{2j} (\eps, x,y, r)] \lesssim x^{\bxi_{2j+1} -\bxi_{2j}} \eps^{\bxi_{2j}}, \text{ provided }  |y|\ge (40)^{2j}r,\\
\label{eq:2j-1}
&\Pb[\Hc^\alpha_{2j-1} (\eps, x,y, r)] \gtrsim x^{\bxi_{2j-1} -\bxi_{2j}} \eps^{\bxi_{2j}}, \text{ provided }  x\asymp r\le  |y|\lesssim r,\\
\label{2j_lo}
&\Pb[\Hc^\alpha_{2j} (\eps, x,y, r)] \gtrsim x^{\bxi_{2j+1} -\bxi_{2j}} \eps^{\bxi_{2j}}, \text{ provided } r\le  |y|\lesssim r,
\end{align}
where the constants in $\lesssim$ and $\gtrsim$ are uniform over $x$ and $\eps$.\footnote{This is the statement given in \cite[Proposition 3.1]{MR3846840}. To be precise, in \eqref{2j_lo}, the constant in $\gtrsim$ can be made uniform over $x$ in a bounded interval $(0, c_0]$, but should depend on $c_0$. Similarly, in \eqref{eq:2j-1}, there are implicit constants $c_1, c_2$ involved in the condition $x\asymp r$. The constant in $\gtrsim$ can be made uniform over $x$ in the interval $c_1r\le x \le c_2 r$, but should at least depend on $c_1$.} 
\end{proposition}

We then recall the following \emph{interior arm event} defined in \cite{MR3846840}.
Let $\eta$ be a chordal SLE$_\kappa$ in $\Hb$ from $0$ to $\infty$.
Fix $z\in \Hb$, $r>0$ and $y\le -4r$.  
Let $\tau_1$ be the first time that $\eta$ hits $B_\eps(z)$.  Let $\Ec_2(\eps,z):=\{\tau_1<\infty\}$.
Let $\sigma_1$ be the first time after $\tau_1$ that $\eta$ hits the connected component containing $y-r$ of $\partial B_r(y) \setminus \eta[0, \tau_1]$. Let $\Ec$ be the event that $z$ is in the unbounded connected component of $\Hb\setminus (\eta([0, \sigma_1]) \cup B_r(y))$. On $\Ec$, let $C_z$ be the connected component of $B_\eps(z)\setminus \eta([0, \sigma_1])$ that contains $z$. Let $C_z^b$ be the unique connected component of $\partial C_z \cap \partial B_\eps(z)$ which can be connected to $\infty$ in $\Hb\setminus  (\eta([0, \sigma_1]) \cup B_\eps(z))$. Let $x_z$ be the ending point of $C_z^b$ if we orient it in the clockwise direction. 
For $j\ge 2$, let $\tau_j$ be the first time after $\sigma_{j-1}$ that $\eta$ hits the connected component of $C_z^b\setminus \eta([0, \sigma_{j-1}])$ containing $x_z$, and let $\sigma_j$ be the first time after $\tau_j$ that $\eta$ hits the connected component of $\partial B_r(y)\setminus \eta([0,\tau_j])$ containing $y-r$. 
For $j\ge 2$, define
\begin{align}\label{eq:int_arm_event}
\Ec_{2j} (\eps, z, y,r) =\Ec\cap \{\tau_j<\infty\}.
\end{align}
Note that the definitions of $\Ec_{2j} (\eps, z, y,r) $ is restrictive on the arcs that each crossing first hits. For instance, the events depicted in Figure~\ref{fig:4arm_event} do not belong to this type of arm events.
In the following Proposition~\ref{prop:sle_int}, there is also some non-explicit constant $R$, as well as an event $\Fc$, which add to the constraints on the $2j$-arm events considered there.

Let us cite the following proposition from  \cite{MR3846840}.
\begin{proposition}[Proposition 4.1, \cite{MR3846840}]
\label{prop:sle_int}
Fix $\kappa\in(0,4]$.\footnote{The statement in \cite[Proposition 4.1]{MR3846840} was made only for $\kappa\in(0,4)$, but we believe that the same argument also works for $\kappa=4$.} 
Fix $z\in \Hb$ with $|z|=1$. 
Define $\Fc=\{\eta[0,\tau_1] \subset B_R\}$. There exists $R>0$ which only depends on $\kappa$ and $z$, such that for $j\ge 1$ and $r,y$ with $R\le r \le (40)^{2j} r\le |y| \lesssim r$, we have 
\begin{align}\label{eq:e2j_wu}
\Pb[\Ec_{2j}(\eps, z, y,r) \cap \Fc] = \eps^{\xi_{2j}+o(1)}.
\end{align}
\end{proposition}

More precisely, in order to establish \eqref{eq:e2j_wu}, \cite[Proposition 4.1]{MR3846840} has proved
\begin{align}
\label{eq:lower}
&\Pb[\Ec_{2j}(\eps, z, y,r) \cap \Fc] \gtrsim \eps^{\xi_{2j}},\\
\label{eq:upper}
&\Pb[\Ec_{2j}(\eps, z, y,r) \cap \Fc] \leq \eps^{\xi_{2j}+o(1)},
\end{align}
where the implicit constant in \eqref{eq:lower} can depend on $\kappa, j, y,r, z$.

\begin{remark}\label{rmk:wu}
In this remark, we point out the aforementioned gap in the proof of the upper bound \eqref{eq:upper} in \cite[Proposition 4.1]{MR3846840} (similar gaps also exist in the proofs of the three upper bounds in \cite[Proposition 3]{MR3768961}, for SLE$_\kappa$ with $\kappa\in (4,8)$).

The proof of the upper bound in \cite[Proposition 4.1]{MR3846840} relies crucially on an induction step, which estimates the probability of $\Ec_{2j}(\eps, z, y,r) \cap \Fc$ using the upper bound \eqref{eq:2j_up} on $\Pb\big[\Hc^\alpha_{2j} (\eps, x, y, r)\big]$ obtained in Proposition~\ref{prop:wu_bdy}, see  \cite[(4.2)]{MR3846840}. 
Let us follow the notations of \cite[Lemma 4.3]{MR3846840}. For $t>0$, let $\Theta_t:=\arg(g_t(z) -W_t)$.
Let $C\ge 16$ be a fixed constant. Let $\xi$ be the first time that $\eta$ hits $\partial B_{C\eps}(z)$. For $\delta\in(0, 1/16)$, let 
\begin{align*}
\wt\Fc=\{\xi<\infty, \Theta_\xi \in(\delta, \pi-\delta), \eta[0,\xi] \subset B_R\}.
\end{align*}
Let $f:=g_\xi -W_\xi$. The proof of \cite[Lemma 4.3]{MR3846840} first establishes that conditionally on $\eta[0,\xi], \wt \Fc$,
\begin{itemize}
\item $f(B_\eps(z))$ is contained in the ball centered at $|f(z)|$ with radius $32 C \eps |f'(z)|/ \delta$, where
\begin{align*}
C \eps |f'(z)| /4 \le |f(z)| \le 8 C \eps |f'(z)|/ \delta,
\end{align*}
\item $f(B_r(y))$ is contained in the ball centered at $f(y)$ with radius $4r f'(y)$, where
\begin{align*}
2y \le f(y) \le y, \quad f'(y) \asymp 1,
\end{align*}
\end{itemize}
and then argues that these two facts, together with the estimates in Proposition~\ref{prop:wu_bdy}, imply that
\begin{align}\label{eq:wu_induction}
\Pb\left[\Ec_{2j+2} (\eps, z, y, r) \mid \eta[0, \xi], \wt\Fc \right] \lesssim \left( C \eps |f'(z)| /\delta\right)^{\bxi_{2j+1}},
\end{align}
where the constant in $\lesssim$ is uniform over $C, \eps, \delta$.
\begin{figure}
	\centering
	\includegraphics[width=0.85\textwidth]{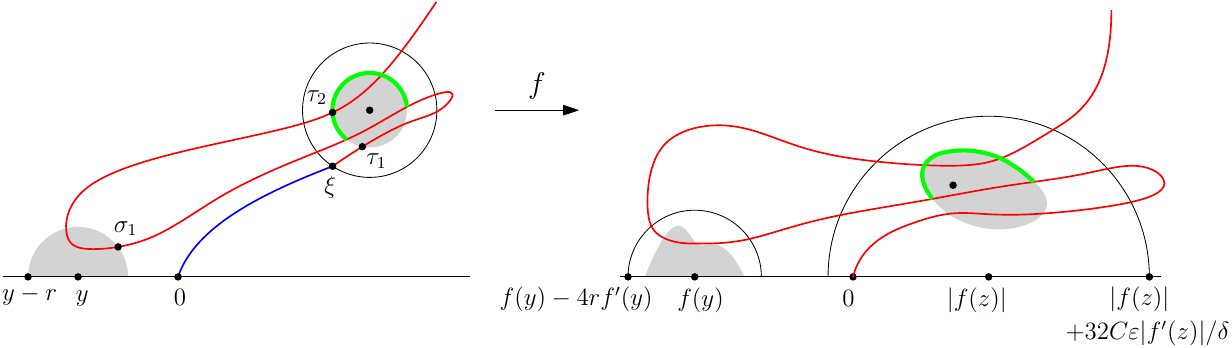}
	\caption{Illustration of Remark~\ref{rmk:wu}. On the left, the event $\Ec_4(\eps, z, y, r)$ occurs for $\eta$. The green arc is $C^b_z$. On the right, the event $\Hc^\alpha_{2} \left(32 C \eps |f'(z)|/ \delta, |f(z)|, f(y), 4 rf'(y)\right)$ does not occur for $f(\eta[\xi, \infty))$.}
	\label{fig:4arm_wu}
\end{figure}

There is no further explanation in  \cite[Lemma 4.3]{MR3846840} about why these facts imply \eqref{eq:wu_induction}, and we believe that these facts alone are not sufficient. Note that conditionally on $\eta[0, \xi], \Fc$, the event $\Ec_{2j+2} (\eps, z, y, r)$ for $\eta$ is \emph{not} included in the event $\Hc^\alpha_{2j} \left(32 C \eps |f'(z)|/ \delta, |f(z)|, f(y), 4 rf'(y)\right)$ for $f(\eta[\xi, \infty))$, due to the restrictive conditions imposed on the events $\Hc^\alpha_{2j}$, see Figure~\ref{fig:4arm_wu} for a counterexample.

We believe that in order to achieve the induction step \eqref{eq:wu_induction}, one needs to prove the same type of upper bound as \eqref{eq:2j_up}, but for a more general boundary arm event than $\Hc^\alpha_{2j}$, where one relaxes the constraints on the topology of the crossings. 
However, the current strategy of proof of \eqref{eq:2j_up} in \cite{MR3846840} is very reliant on the choice of \emph{one} specific arc that the SLE curve hits each time, because it is difficult to control the images of the \emph{infinitely} many arcs (i.e.\ connected components) of $\partial B_r(y) \setminus \eta([0,\wt\sigma_j])$ or $\partial B_\eps(x) \setminus \eta([0, \wt\sigma_j])$ under the relevant conformal maps, in particular because there can be infinitely many arcs on both sides of $\eta([0, \wt\sigma_j])$.
\end{remark}

\subsection{Interior four-arm event}\label{subsec:proof_thm}

As mentioned earlier, the lower bound in \eqref{eq:sle2} is implied by the lower bound \eqref{eq:lower}.
We thus focus on proving the up-to-constant upper bound in \eqref{eq:sle2}, which we do by contradiction. 
Recall the definition of the interior four-arm event $\Wc_4 (a, \eps, r)$ given in \eqref{eq:def_4arm_int}.
We make the following assumption.
\begin{ass}\label{ass}
There exist $a_0\in\partial\Db\setminus \{1\}$ and $r_0\in(0,1)$, such that for all $C>0$ and $\eps_0>0$, there exists $\eps\in (0, \eps_0)$, such that
$\Pb[\Wc_4 (a_0, \eps, r_0)] \ge C \eps^{\xi_4}$.
\end{ass}
We produce a contradiction under Assumption~\ref{ass}, which will prove the upper bound in Theorem~\ref{thm:sle_int_utc}.

\begin{lemma}\label{lem:assump}
Suppose that Assumption~\ref{ass} holds. Then for all $a\in\partial\Db\setminus \{1\}$ and $r\in(0,r_0/8]$, for all $C>0$ and $\eps_0>0$, there exists $\eps\in (0, \eps_0)$, such that
$\Pb[\Wc_4 (a, \eps, r)] \ge C \eps^{\xi_4}$.
\end{lemma}

\begin{figure}
	\centering
	\includegraphics[width=0.8\textwidth]{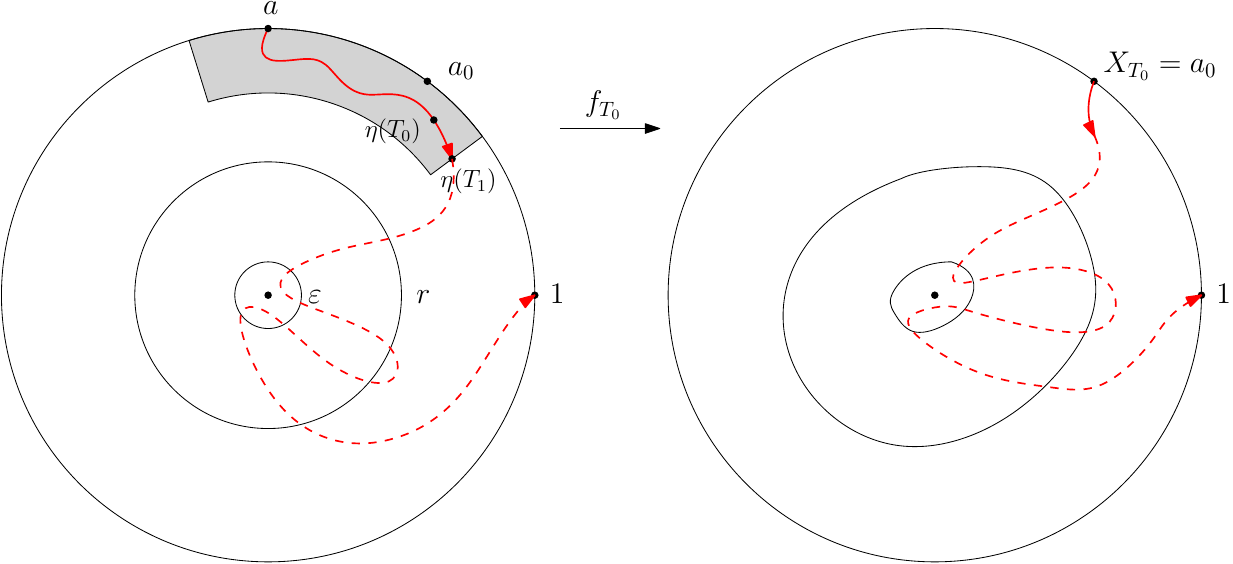}
	\caption{Illustration of the proof of Lemma~\ref{lem:assump}. The grey area is $U$.}
	\label{fig:4arm_SLE}
\end{figure}

\begin{proof}
Fix $C>0$ and $\eps_0>0$.
Without loss of generality, we can suppose that $a$ is on the clockwise arc from $1$ to $a_0$, see Figure~\ref{fig:4arm_SLE} (left). If $a$ is on the counterclockwise arc from $1$ to $a_0$, then it suffices to consider the mirror symmetry of this picture w.r.t.\ the real axis. Note that the image of a chordal SLE curve under such a mirror symmetry is still a chordal SLE (whose endpoints are the images of the endpoints of the original curve).

Suppose $a_0=\exp(i\theta_0)$ and $a=\exp(i\theta)$ for $0<\theta_0< \theta <2\pi$. Let $\delta_0\in (0, \min(\theta_0, 2\pi-\theta)/4)$. Let $\delta_1 \in (0, 1/2)$  be a sufficiently small quantity that we will adjust later. Let $U$ be the area bounded between the arcs $\{ e^{it},  \theta_0-\delta_0 \le t\le \theta+\delta_0\}$, $\{(1-\delta_1)e^{it},  \theta_0-\delta_0 \le t\le \theta+\delta_0\}$ and the segments $\{t e^{i(\theta_0-\delta_0)}, 1-\delta_1\le t \le 1\}$, $\{t e^{i(\theta+\delta_0)}, 1-\delta_1\le t \le 1\}$. 
Let $\eta$ be the SLE$_\kappa$ in $\Db$ from $a$ to $1$.
Let $E_0$ be the event that $\eta$ first exits $U$ through the segment $S_0:=\{t e^{i(\theta_0-\delta_0)}, 1-\delta_1\le t \le 1\}$. 
We claim that there exists $p_0>0$ which depends only on $\delta_0, \delta_1, a_0$ such that $\Pb(E_0)\ge p_0$.
Indeed, by Lemma~\ref{lem:sle_tube}, we can let $\eta$ stay in some given tube with positive probability.
We can choose this tube in a way that if $\eta$ stays in this tube, then $E_0$ occurs.

We parametrize $\eta$ according to its radial capacity, namely $|f_t'(0)|=e^t$, where $f_t$ is the conformal map from $\Db\setminus \eta([0,t])$ onto $\Db$ which leaves $0,1$ fixed.
Let $T$ be its total time length, and $T_1$ be the first time that $\eta$ exits $U$.
 Let $X_t$ denote the image under $f_t$ of the tip $\eta(t)$.
Let $T_0$ be the first time $t$ that $X_{t}=a_0$. Let us now prove that on $E_0$, we have $T_0<T_1$ a.s. 
Note that $X_0=a$ and $X_t$ is continuous in $t$. Since $X_t\not=1$ for all $t\in [0, T_1]$, it suffices to show that $X_{T_1}$ is on the clockwise arc from $a_0$ to $1$. To see that latter point, let $S$ be  the segment from $e^{i(\theta_0-\delta_0)}$ to  $\eta(T_1)$ and let $A$ be the clockwise arc from $e^{i(\theta_0-\delta_0)}$ to $1$.
The length of the clockwise arc from $X_{T_1}$ to $1$ is less than $2\pi \omega$, where $\omega$ is the harmonic measure of $S\cup A$ seen from $0$ in the domain $\Db\setminus (\eta([0, T_1]) \cup S)$.
Indeed, the preimage under $f_{T_1}$ of a Brownian motion started at $0$ which exits $\Db$ through  the clockwise arc from $X_{T_1}$ to $1$ must first exit $\Db\setminus (\eta([0, T_1]) \cup S)$ through $S \cup A$. 
Note that $\omega$ tends to $\theta_0-\delta_0$ as $\delta_1$ tends to $0$. Hence we can choose $\delta_1$ sufficiently small so that $\omega\in(0,\theta_0)$, 
which guarantees that $X_{T_1}$ is on the clockwise arc from $a_0$ to $1$. This ensures that $T_0<T_1$ a.s.\ on $E_0$. Therefore $\Pb(T_0<T_1) \ge p_0$.

Let $(\Fc_t)_{t\ge 0}$ be the filtration adapted to $\eta$. For all $r\in (0, r_0/8]$, we have
\begin{align}\label{eq:w41}
\Pb[\Wc_4 (a, \eps, r)] \ge \Pb[\Wc_4 (a, \eps, r) \one_{T_0<T_1}] =  \Eb \big[ \Pb[\Wc_4 (a, \eps, r) \mid \Fc_{T_0} ]\, \one_{T_0<T_1} \big].
\end{align}
Note that $f_{T_0}(\eta([T_0,T]))$ is distributed as a SLE$_\kappa$ in $\Db$ from $a_0$ to $1$.
On the event $T_0<T_1$, we have $1\le f_{T_0}'(0)\le 2$ by the Schwarz lemma. Therefore, by the Koebe $1/4$ theorem, we have
$$
f_{T_0} (B_{r_0/8}) \subseteq B_{r_0}, \quad f_{T_0} (B_\eps) \supseteq B_{\eps/4}.
$$
Therefore
\begin{align}\label{eq:w42}
\Pb[\Wc_4 (a, \eps, r_0/8) \mid \Fc_{T_0} ] \ge \Pb[\Wc_4(a_0, \eps/4, r_0)]
\end{align}
By Assumption~\ref{ass}, there exists $\eps \in (0, \eps_0)$ such that $\Pb[\Wc_4(a_0, \eps/4, r_0)] \ge (C/p_0) \eps^{\xi_4}$. Injecting this back into \eqref{eq:w41} and \eqref{eq:w42}, we get
\begin{align*}
\Pb[\Wc_4 (a, \eps, r_0/8)] \ge (C/p_0) \eps^{\xi_4} \Pb[T_0<T_1] \ge C  \eps^{\xi_4}.
\end{align*}
Since $\Pb[\Wc_4 (a, \eps, r)]$ is increasing as $r$ decreases, we have completed the proof.
\end{proof}

\begin{proof}[Proof of \eqref{eq:sle2} in Theorem~\ref{thm:sle_int_utc}]
We explore $\Gamma$ according to Exploration process~\ref{expl2}, and use the notations therein. We fix $r=1/2$ for the quantity $r$ in Exploration process~\ref{expl2}.
On the event $E_1$, by the Schwarz lemma, we have $1\le f'(0) \le 2/r=4$.
By the Koebe $1/4$ theorem, for $s\in (0, 1/32)$, we have
\begin{align}\label{eq:koebe2}
B_{\eps/4} \subseteq f(B_\eps), \quad f(B_s) \subseteq B_{16s}. 
\end{align}
Let $E_2$ be the event that $\wt \gamma$ (seen as a curve from $f(a)$ to $f(b)$) first intersects $B_{\eps/4}$, then intersects $\partial B_{16s}$, then intersects $B_{\eps/4}$ again. If both $E_1$ and $E_2$ hold, then $\ione(\eps, s)$ occurs. Therefore
\begin{align}\label{eq:int_lo1-2}
\Pb\big[\ione(\eps, s)\big] \ge \Pb[E_1 \cap E_2] =\Eb[ \Pb[E_2 \mid \Sigma] \one_{E_1}].
\end{align}
On $E_1$, conditionally on $\Sigma$, $\wt \gamma$ is a chordal SLE$_\kappa$ in $\Hb$ between $f(a)$ and $f(b)$. 
We now try to induce a contradiction under Assumption~\ref{ass}.
By Lemma~\ref{lem:assump}, for $s \in (0, r_0/128)$, for all $C>0$ and $\eps_0>0$, there exists $\eps\in(0, \eps_0)$, such that
$
\Pb[E_2 \mid \Sigma] \ge C \eps^{\xi_4}.
$
Injecting it back into \eqref{eq:int_lo1-2} leads to
\begin{align}\label{eq:pe1low}
\Pb\big[\ione(\eps, s)\big] \ge C \Pb[E_1]  \eps^{\xi_4}.
\end{align}
On the other hand, by \eqref{eq:i4arm}, there exists $c_0>0$ (which depends on $s, \eps_0$, but not on $C$ or $\eps$) such that $\Pb\big[\ione(\eps, s)\big]\le c_0 \eps^{\xi_4}$. Combining it with \eqref{eq:pe1low} yields
\begin{align*}
c_0 \eps^{\xi_4} \ge C \Pb[E_1]  \eps^{\xi_4}.
\end{align*}
This leads to a contradiction since $\Pb[E_1]>0$ and we can choose $C$ as big as we want. This completes the proof of \eqref{eq:sle2}.
\end{proof}

\subsection{Boundary four-arm event}\label{subsec:proof_thm_bdy}

We now prove the upper bound in \eqref{eq:sle1}. 
We will prove the following lemma, which then implies  \eqref{eq:sle1}.
Let us first define the following event.
For $y>x>1>s>\eps>0$, let $\eta$ be a chordal SLE$_\kappa$ in $\Hb$ from $-x$ to $-y$. Let $\tau_1$ be the first time that $\eta$ hits $B_\eps$. 
Let $\sigma_1$ be the first time after $\tau_1$ that $\eta$ hits $\partial B_s$.
Let $\tau_2$ be the first time after $\sigma_1$ that $\eta$ hits $B_\eps$. Let
\begin{align*}
\wt\Wc^+_4 (x,y, \eps,s):=\{\tau_2<\infty\}.
\end{align*}

\begin{lemma}\label{lem:up}
Fix $\kappa\in(8/3,4]$. For each $s\in(0,1)$, there exists $y_0>x_0>1$, such that as $\eps\to 0$, 
\begin{align}\label{eq:wtwc}
\Pb\big[\wt\Wc^+_4 (x_0,y_0, \eps,s)\big]\lesssim \eps^{\bxi_4(\kappa)},
\end{align}
where the implicit constant depends on $\kappa, s$.
\end{lemma}
\begin{proof}
We explore $\Gamma^+$ according to Exploration process~\ref{expl1}, and use the notations therein.
We fix $r=2$, for the quantity $r$ in Exploration process~\ref{expl1}.
On the event $E_1$, by Lemma~\ref{lem:Koebe}, for $s\in(0,1)$, we have
\begin{align}\label{eq:koebe2}
B_{\eps/4} \cap \Hb \subseteq f(B_\eps \cap \Hb), \quad f(B_s \cap \Hb) \subseteq B_{4s}\cap \Hb. 
\end{align}
Let $E_2$ be the event that $\wt \gamma$ (seen as a curve from $f(a)$ to $f(b)$) first intersects $B_{\eps/4}$, then intersects $\partial B_{4s}$, then intersects $B_{\eps/4}$ again. If both $E_1$ and $E_2$ hold, then $\bone(\eps, s)$ occurs. Therefore
\begin{align}\label{eq:int_lo1-1}
\Pb\big[\bone(\eps, s)\big] \ge \Pb[E_1 \cap E_2] =\Eb[ \Pb[E_2 \mid \Sigma] \one_{E_1}].
\end{align}
On $E_1$, conditionally on $\Sigma$, $\wt \gamma$ is a chordal SLE$_\kappa$ in $\Hb$ between $f(a)$ and $f(b)$.

Suppose that the lemma is not true, then  there exists $s_0\in(0,1)$, such that for all $y>x>1$, $C>0$ and $\eps_0>0$, there exists $\eps\in (0, \eps_0)$, such that
$\Pb\big[\wt\Wc^+_4 (x,y, \eps/4,s_0)\big]> C \eps^{\bxi_4(\kappa)}.$
Letting $s:=s_0/4$, $y:=-f(a)$, $x:=-f(b)$, we have
\begin{align*}
\Pb[E_2 \mid \Sigma] =\Pb \big[\wt\Wc^+_4 (-f(b),-f(a), \eps/4, 4s) \big]> C \eps^{\bxi_4(\kappa)}.
\end{align*}
Putting it back into \eqref{eq:int_lo1-1}, we get
\begin{align}\label{eq:bpe1low}
\Pb\big[\bone(\eps, s)\big] \ge C \eps^{\bxi_4(\kappa)}\Pb[E_1].
\end{align}
On the other hand, by \eqref{eq:b4arm}, there exists $c_0>0$ (which depends on $s, \eps_0$, but not on $C$ or $\eps$) such that $\Pb\big[\bone(\eps, s)\big]\le c_0 \eps^{\xi_4}$. Combining it with \eqref{eq:bpe1low} yields
\begin{align*}
c_0 \eps^{\xi_4} \ge C \Pb[E_1]  \eps^{\xi_4}.
\end{align*}
This leads to a contradiction since $\Pb[E_1]>0$ and we can choose $C$ as big as we want. This completes the proof.
\end{proof}

We are now ready to prove the upper bound in \eqref{eq:sle1}.
\begin{proof}[Proof of the upper bound in \eqref{eq:sle1}]
We will prove that for $1>r>\eps>0$,
\begin{align}
\Pb[\Wc^+_4 (\eps,r)] \lesssim \eps^{\bxi_4(\kappa)},
\end{align}
where the implicit constant depends on $\kappa,r$.
Suppose that $\eta$ is a chordal SLE$_\kappa$ in $\Hb$ from $0$ to $\infty$. For $y>x>1$, let $\wt \eta$ be the image of $\eta$ under the conformal map 
$f(z)=(xyz-xy)/(-xz+y).$
Then $\wt \eta$ is a SLE$_\kappa$ in $\Hb$ from $-x$ to $-y$. Note that $f$ sends the half circle $\partial B_r(1) \cap \Hb$ (resp.\ $\partial B_\eps(1) \cap \Hb$) to the half circle with endpoints $f(1-r)$ and $f(1+r)$ (resp.\ $f(1-\eps)$ and $f(1+\eps)$). 
Let
\begin{align*}
r_0:= \min (|f(1-r)|, |f(1+r)|)= \frac{xy}{y-x(1-r)}r, \, \eps_0=\max(|f(1-\eps)|, |f(1+\eps)|)= \frac{xy}{y-x(1+\eps)}\eps.
\end{align*}
Note that $r_0\ge r$. If $\Wc^+_4 (\eps,r)$ occurs for $\eta$, then $\wt\Wc^+_4 (x,y,\eps_0,r)$ occurs for $\wt\eta$.
For $\eps$ sufficiently small, we have $\eps_0\le 2\eps xy/ (y-x)$.
Let $x:=x_0$ and $y:=y_0$ for $x_0, y_0$ chosen according to Lemma~\ref{lem:up} for $s=r$. Then we have
\begin{align*}
\Pb\big[ \Wc^+_4 (\eps,r)\big] \le \Pb\big[ \wt\Wc^+_4 (x_0, y_0, \eps_0,r)\big] \lesssim \eps^{\bxi_4(\kappa)},
\end{align*}
where the implicit constant depends on $\kappa,r$. This completes the proof.
\end{proof}

\begin{proof}[Proof of Theorem~\ref{thm:sle_int_utc}]
First, observe that the upper bound in \eqref{eq:sle2} is proved in Section~\ref{subsec:proof_thm}, while the upper bound in \eqref{eq:sle1} is proved earlier in this section. 

As for the lower bound in \eqref{eq:sle1}, note that $\Hc^\alpha_{3} (\eps, 1,-4, 1)$ is contained in $\Wc^+_4 (\eps,r)$ for all $1>r>\eps>0$. Then \eqref{eq:2j-1} implies that $\Pb[\Wc^+_4 (\eps,r)] \gtrsim \eps^{\bxi_4}$, where the implicit constant depends only on $\kappa$.

In order to get the lower bound in \eqref{eq:sle2}, we send $\Db$ onto $\Hb$ by the unique conformal map $f$ which sends $1,a$ to $0,\infty$ with $|f'(0)|=1$. Let $z=f(0)$.
Let $\eta$ be an SLE$_\kappa$ in $\Hb$ from $0$ to $\infty$. Let $R>0$ be the constant in Proposition~\ref{prop:sle_int} and let $\Fc$ be the event defined therein. 
We can choose $y$ sufficiently large, so that $y>40^4R$ and $B_R(y) \cap f(B_r)=\emptyset$. 
By Koebe $1/4$ theorem, we also have $B_{\eps/4}(z) \subseteq f(B_\eps)$.
Therefore, if the event $\Ec_{4}(\eps/4, z,y, R) \cap \Fc$ holds for $\eta$, then $\Wc_4 (a, \eps, r)$ holds for $f^{-1}(\eta)$. Then \eqref{eq:lower} implies that $\Pb[\Wc_4 (a, \eps, r)]\gtrsim \eps^{\xi_4}$, where the implicit constant depends on $\kappa, a, r$.
\end{proof}

\subsection*{Acknowledgments}

YG and PN are partially supported by a GRF grant from the Research Grants Council of the Hong Kong SAR (project CityU11307320). WQ is partially supported by a GRF grant from the Research Grants Council of the Hong Kong SAR (project CityU11308624).

\bibliographystyle{abbrv}
\bibliography{cr}

\end{document}